\documentclass[12pt,a4paper,final]{article}

\usepackage[margin=25mm]{geometry}
\usepackage{amsmath,amsfonts,amssymb}
\usepackage{bef_alex,bm}
\usepackage[utf8]{inputenc}

\usepackage[notref,notcite,color]{showkeys}

\usepackage{color}
\usepackage{graphicx,todonotes}

\numberwithin{equation}{section}
\numberwithin{figure}{section}

\definecolor{ddmagenta}{rgb}{0.7,0,0.9}
\definecolor{ddorange}{rgb}{0.8,0.2,0}
\definecolor{ALEXBLAU}{rgb}{0.4,0.0,0.9}

\newcommand{\AAA}{}

\newcommand{\EEE}{\color{black}}

\newcommand{\FHN}{\mathrm{FN}}
\newcommand{\bbGamma}{\rmI\hspace{-0.13em}\Gamma}
\newcommand{\ttc}{\mathtt{c}}
\newcommand{\ttv}{\mathtt{v}}
\newcommand{\ttvminus}{\mathtt{v}_{\!\bm-}}
\newcommand{\ttC}{\mathtt{C}}
\renewcommand{\ti}{{\times}}

\renewcommand{\div}{\mathop{\mathrm{div}}\nolimits}
\makeatletter
\newcommand{\xrightharpoonup}[2][]{\ext@arrow 0359{%
        \arrowfill@\relbar\relbar\rightharpoonup}{#1}{#2}}
\makeatother
\DeclareMathOperator{\wtc}{\xrightharpoonup{2w}}
\DeclareMathOperator{\stc}{\xrightarrow{2s}}

\begin{document}

\title{Traveling fronts in a reaction-diffusion equation \\
 with a memory term\thanks{Research partially supported by Deutsche
   Forschungsgemeinschaft via subproject A5 within SFB 910 \emph{Control of
  self-organizing nonlinear systems} (Project no.\ 163436311).}} 

\author{Alexander Mielke$^{*,\dagger}$ and Sina Reichelt$^*$\\
 \begin{minipage}{0.8\textwidth} \normalsize 
  $^*$ Weierstraß-Institut für Angewandte Analysis und Stochastik, Berlin\\
  $^\dagger$ Institut für Mathematik, Humboldt-Universität zu Berlin
 \end{minipage}
}

\date{26. April 2021}
\maketitle

\centerline{\large\itshape In memory of Pavel Brunovsk\`y}\medskip

\begin{abstract}
  Based on a recent work on traveling waves in spatially nonlocal
  reaction-diffusion equations, we investigate the existence of traveling
  fronts in reaction-diffusion equations with a memory term. We will explain how
  such memory terms can arise from reduction of reaction-diffusion systems if the
  diffusion constants of the other species can be neglected. In particular, we
  show that two-scale homogenization of spatially periodic systems can induce
  spatially homogeneous systems with temporal memory.

  The existence of fronts is proved using comparison principles as well as
  a reformulation trick involving an auxiliary speed that allows us to
  transform memory terms into spatially nonlocal terms. Deriving explicit
  bounds and monotonicity properties of the wave speed of the arising traveling
  front, we are able to establish the existence of true traveling fronts for
  the original problem with memory. Our results are supplemented by
  numerical simulations.
\end{abstract}

\section{Introduction}
\label{s:Intro}

The study of traveling fronts in scalar reaction-diffusion equations with
a bistable nonlinearity is a classical topic and there is a rich literature
concerning the existence, uniqueness, and asymptotic stability of such fronts,
see e.g.\ \cite{FifMcl77ASND} and \cite{Chen97EUAS, AchKue15TWBE} and the
references therein. The last two references are actually devoted to scalar,
spatially nonlocal, parabolic equations of the form
\begin{equation}
  \label{eq:I.NonlocPDE}
  \dot u = L u + F(u) + G\big(u, \bbJ S(u)\big), \quad t>0 , \ x\in \R,
\end{equation}
where traveling fronts of the form $u(t,x)= \calU(x{-}c t) \in \R$ are
investigated.  In particular, \cite{Chen97EUAS} studies the case $Lu=Du_{xx}$,
assumes that $G(u,\cdot):\R\to \R$ and $S:\R\to \R$ are monotonously increasing
and that $\bbJ$ is a convolution with a smooth, nonnegative kernel $J\geq 0$,
i.e.\ $(\bbJ f)(x)=(J{*}f) (x)=\int_\R J(x{-}y)f(y)\dd y$. The paper \cite{AchKue15TWBE}
considers the case $G\equiv 0$ and $Lu =\rmD^\alpha_\theta u$, which is a
Riesz-Feller operator with $\alpha \in {]1,2]}$ and
$|\theta| \leq \min\{\alpha,2{-}\alpha\}$. In both cases, the nonlinearity $F$
is of bistable type such that \eqref{eq:I.NonlocPDE} has exactly three
homogeneous steady states $u\equiv u_\alpha$ with $\alpha\in \{-,\rmm,+\}$, where
$u_\pm$ are stable while the middle state $u_\rmm$ is unstable.

In \cite{Chen97EUAS, AchKue15TWBE} it is shown that, under suitable technical
assumptions, equation \eqref{eq:I.NonlocPDE} admits traveling fronts
$u(t,x)=\calU(x{-}c t)$ with $\calU(\xi)\to u_\pm$ for $\xi \to \pm \infty$,
and that these solutions are even stable up to translation. The crucial tool for
the analysis in these two papers are comparison principles that are flexible
enough to survive nonlocal operators with positive kernels and monotone
operators. The only non-monotonicity occurs in the local function
$u \mapsto F(u)$.

Here we treat a similar type of equation but allow for memory terms
(nonlocality in time), namely the equation
\begin{equation}
  \label{eq:I.Memory}
   \dot u(t,x) = D u_{xx}(t,x) + F(u(t,x)) + \int_0^\infty 
   \wh\Gamma(\tau) \, u(t{-}\tau,x) \dd \tau , \quad t>0 , \ x\in \R,
\end{equation}
with an integrable, nonnegative memory kernel $\wh \Gamma$. A more general,
nonlinear memory kernel of the type $G\big(u, \wh\Gamma{*} S( u(\tau,x))\big)$
could be treated as well (cf.\ Section \ref{su:Generalizations}), 
but we avoid this complexity to keep the main
arguments simple. Of course, memory terms of this type destroy any kind of
comparison principles, so the ideas in \cite{Chen97EUAS, AchKue15TWBE} cannot
be applied directly. Traveling fronts in parabolic equations with discrete time
delay are treated in \cite{WuZou01TWFR}, but not including the bistable case
considered here.  

However, introducing the auxiliary wave speed $\ttv \in \R$, we look at a
corresponding equation with spatial nonlocality, namely  
\begin{equation}
  \label{eq:I.ttv}
   \dot u(t,x) = D u_{xx}(t,x) + F(u(t,x)) + \int_0^\infty 
   \wh\Gamma(\tau)\, u(t, x{+}\ttv \tau) \dd \tau , \quad t>0 , \ x\in \R.
\end{equation}
This nonlocal equation has now traveling fronts if the function
$\wh F: u \mapsto F(u) + \gamma u$ with
$\gamma=\int_0^\infty\wh\Gamma(\tau)\dd \tau$ is a bistable nonlinearity and the
associated fronts will have a wave speed $c=C(\ttv)$. The main observation is
now that in the special case $\ttv=C(\ttv)$ the traveling wave
$u(t,x)=\calU(x{-}\ttv t)$ of \eqref{eq:I.ttv} also solves the original memory
equation \eqref{eq:I.Memory}, see Proposition \ref{pr:TimeSpace}. 

Hence, existence of traveling fronts for the memory equation \eqref{eq:I.Memory}
will be established by providing conditions that guarantee that the function
$\ttv \mapsto C(\ttv)$ has a fixed point. Indeed, we again use the comparison
principle for the spatially nonlocal equation \eqref{eq:I.ttv} to derive
continuity of $C$ and certain monotonicities. In Theorem \ref{th:MainResult} we can
conclude that the memory equation \eqref{eq:I.Memory} has always a traveling
front if $\wh F: u \mapsto F(u) + \gamma u$ is a bistable nonlinearity.  
In Corollary \ref{co:SpeedBounds} we provide bounds on the associated wave
speed in terms of the wave speed of the local equation with nonlinearity $\wh
F$. 

In Section \ref{se:Derivation} we discuss several derivations of the memory
equation from classical reaction-diffusion systems. In all cases, the memory
appears by a coupling to ODEs or PDEs acting locally in $x \in \R$ but induce
memory through its internal dynamics. The simplest case occurs in the coupled
PDE-ODE system considered extensively in \cite{GurRei18PFNS}:
\begin{equation}
  \label{eq:I.PDE-ODE-System}
  \dot u = Du_{xx} + F(u) + \sum_{i=1}^{m_\rmm} a_i w_i, \qquad \dot w_i = -
  \lambda_i w_i + b_i u \ \text{ for }i=1,\ldots,m_\rmm,
\end{equation}
where $a_i, \ b_i$, and $\lambda_i$ are fixed real parameters with
$\lambda_i>0$. Clearly, the linear ODEs can be solved by $w_i(t)=\int_0^\infty
\ee^{-\lambda_i \tau} b_i u(t{-}\tau)\dd \tau$, and we obtain
\eqref{eq:I.Memory} with 
\[
\wh\Gamma(\tau)= \sum_{i=1}^{m_\rmm} \ee^{-\lambda_i \tau} a_i b_i, \quad
\text{ giving }\gamma =\int_0^\infty\!\wh\Gamma(\tau)\dd\tau= \sum_{i=1}^{m_\rmm}  \frac{a_ib_i}{\lambda_i} . 
\] 
We emphasize that our theory is applicable also in cases where some of the
products $a_ib_i$ are negative, as long as $\wh \Gamma$ is nonnegative. 
We also discuss possible nonlinear couplings that still allow for the
application of \cite{Chen97EUAS}. 

A second motivation for deriving memory equations is the study of traveling
pulses and fronts in situations where the coefficients in the system are
rapidly oscillating, thus modeling \EEE a periodic heterostructure of the
medium on a microscopic scale. In the following, the homogenization parameter
$\eps>0$ denotes the ratio between the characteristic length scales of the
microscopic and macroscopic structure. \AAA A typical system considered in
Section \ref{s:TwoscaleHom} is the heterogeneous FitzHugh-Nagumo system  
\begin{equation}
  \label{eq:I.oscFHN}
\begin{aligned}
\dot u&=\big( D_\rmu(\tfrac1\eps x) u_x\big)_x + F(\tfrac1\eps x, u)
-\alpha\big(\tfrac1\eps x \big) w, 
\\
\dot w & = \big( \eps^2 D_\rmw(\tfrac1\eps x) w_x\big)_x - \lambda
\big(\tfrac1\eps x \big)w + \beta \big(\tfrac1\eps x \big) u, 
\end{aligned}
\end{equation}
where all dependence on $\tfrac1\eps x$ is assumed to be 1-periodic. Of
course, exact traveling waves cannot exist, but one still expects periodically
modulated pulses and fronts traveling through the heterogeneous medium. Indeed,
\EEE there exists a vast literature on the study of traveling fronts for
reaction-diffusion equations in continuous periodic
media, e.g.\ \cite{HudZin95ETWR, BerHam02FPPE}, fronts in discrete periodic media
\cite{GuoHam06FPDP, ChGuWu08TWDP}, and fronts in perforated domains \cite{Hein01WSRD}. We
refer to \cite{Xin00FPHM} for a review \AAA including \EEE references to earlier
works. Most of the latter results share a common approach based on comparison
principles, and so do we. However, our approach also allows for systems of
reaction-diffusion equations.

In \cite{MaScUe07EATW} reaction-diffusion systems are studied and exponential
averaging is used to show that traveling wave solutions can be described by a
spatially homogeneous equation and exponentially small remainders. The approach
based on center-manifold reduction in \cite{BodMat14EHTW} applies to traveling
waves in parabolic equations and, moreover, the authors prove the existence of
a generalized oscillating wave that converges to a limiting wave. We point out
that all previously mentioned articles study limit problems of ``one-scale''
nature, in contrast to \cite{GurRei18PFNS} and the present work where
traveling pulses in ``two-scale FitzHugh--Nagumo systems'' are
investigated, see Section \AAA \ref{su:ReductScalar}. According to Theorem
\ref{thm:ts-conv-2}, the solutions $(u_\eps,w_\eps)$ of \eqref{eq:I.oscFHN}
converge to the solution $(U,W)$ of the two-scale system
 \begin{equation}
  \label{eq:I.tsFHN}
\begin{aligned}
\dot U(t,x)&=  D_\rmu^\text{eff} U_{xx}(t,x) + \int_\bbT F(y, U(t,x))\dd y
-\int_\bbT \alpha(y) W(t,x,y) \dd y, 
\\
\dot W(t,x,y) & = \big( D_\rmw(y) W_y(t,x,y)\big)_y - \lambda
(y)W(t,x,y) + \beta (y) U(t,x),
\end{aligned}
\end{equation}
where the microscopic variable $y$ lies in the circle $\bbT=\R_{\!/\Z}$. The
importance of the two-scale system \eqref{eq:I.tsFHN} is that the microscopic
variable $y$ decouples from the macroscopic space variable $x$ such that this
system admits exact traveling waves of the form
$(U(x{-}ct),W(x{-}ct,y))$. Moreover, the coupling from $U$ to $W$ and its
feedback are local in $x\in \R$ and hence can be captured in spatially local
memory terms as in \eqref{eq:I.ttv}.

The paper is structured as follows. In Section \ref{se:ExistFront} we develop
our existence result for traveling fronts for the memory equation
\eqref{eq:I.Memory}, see Theorem \ref{th:MainResult}  which relies on the
comparison principle in the spirit of \cite{Chen97EUAS, AchKue15TWBE} and on new a
priori bounds for the front speed $c$ in Theorem \ref{thm:SpeedBounds}. \EEE In
Section \ref{se:Derivation} we first present the derivation of memory equations
from coupled, but homogeneous systems, and secondly we show that two-scale
homogenization can be used to derive effective two-scale limits, which again lead to
the same memory equation. Finally, Section \ref{su:CubicCase} compares the
abstract theory with numerical results in the special case that $F$ is the
classical bistable cubic polynomial $-u(u{-}a)(u{-}1)$ and the memory
kernel is simply given by $\wh\Gamma(\tau)=-\beta\,\ee^{-\tau}$.

\section{Existence of traveling fronts}
\label{se:ExistFront}

We first describe the setup and the assumptions for our theory in Section
\ref{su:SetupAss}. In Section \ref{su:aux-prob} we introduce the spatially
nonlocal equation with the auxiliary speed $\ttv$ and show that for traveling
fronts these equations are related. The main technical part are Sections
\ref{sec:comparison} and \ref{su:FrontsAux} where we exploit the comparison
principles developed in \cite{Chen97EUAS}. The main results are presented in
Section \ref{su:Results}, where we also discuss potential generalizations.

\subsection{Setup of the memory equation and assumptions}
\label{su:SetupAss}

\AAA In all of Section \ref{se:ExistFront} we study the Nagumo-type
reaction-diffusion equation (cf.\ \cite{NaArYo62APTL,Mcke70NE})
with linear memory term in the form
\begin{equation}
  \label{eq:NaguMemo}
  \dot u(t,x) = D u_{xx}(t,x) + F(u(t,x)) + \gamma \int_0^\infty \Gamma(\tau)
  u(t{-}\tau,x) \dd \tau .  
\end{equation}
Here, the nonnegative memory kernel $\Gamma$ is normalized, and hence satisfies 
\begin{equation}
  \label{eq:GammaConds}
  \Gamma \in \rmC^0({[0,\infty[}), \quad \Gamma(\tau)\geq 0 \text{ for all
  }\tau\geq 0, \quad \int_0^\infty \Gamma(\tau)\dd \tau=1. 
\end{equation}
To formulate the precise  conditions on the bistable nonlinearity  $F$ we
introduce the tilted function 
\[
F_\gamma (u)= F(u)+ \gamma u \quad \text{for } u\in \R. 
\]

\noindent\textbf{Hypotheses (H):}
\begin{itemize}
  \itemsep-0.3em
\item[\bfseries(H1)] We have $F\in \rmC^2(\R)$ and there exists $\gamma_*>0$
  such that for all $\gamma\in [0,\gamma_*]$ the function $F_\gamma$ \EEE has
  exactly three zeros, which we assume to be
  $u_-^\gamma < u_\rmm^\gamma < u_+^\gamma$ and which satisfy
  $F_\gamma'(u_-^\gamma )<0$, $F'(u^\gamma_\rmm)>0$, and $ F_\gamma'(u^\gamma_+)< 0$.

\item[\bfseries(H2)] The memory kernel $\Gamma$ satisfies \eqref{eq:GammaConds}
  \AAA and $\wh g_1:=\int_0^\infty \Gamma(\tau)\tau \dd \tau <\infty$. \EEE
\end{itemize} 

The assumptions imply that the three constant functions $u(t,x)= u^\gamma_\alpha$
with $\alpha \in \{-,\rmm,+\}$ are indeed trivial solutions for
\eqref{eq:NaguMemo}. The solutions $u_\pm^\gamma$ will be stable, while the middle
solution $u^\gamma_\rmm$ is unstable. Our aim is to show the existence of nontrivial
traveling fronts $u(t,x)=\calU(x{-}c t)$ connecting the two stable levels
$u^\gamma_\pm$, namely $\calU(\xi)\to u^\gamma_\pm$ for $\xi \to \pm \infty$. Of course,
then also a reflected traveling front $u(t,x) = \calU^\text{refl} 
(x{-}c^\text{refl}t)$ exists satisfying $ \calU^\text{refl}(\xi)\to u^\gamma_\mp$
for $\xi \to \pm \infty$. Indeed, using the reflection symmetry $x \to -x$ of
\eqref{eq:NaguMemo} gives $\calU^\text{refl}(\xi) =\calU({-}\xi)$ and
$c^\text{refl}= - c$.  \EEE

\subsection{Traveling waves for the auxiliary equation}
\label{su:aux-prob}

The works in \cite{Chen97EUAS,AchKue15TWBE} allow for nonlocal terms
in the reaction-diffusion equation, but only for spatial nonlocality
and not for temporal one. However, for traveling waves space and time
coincide up to a scaling, so we look at an auxiliary problem, where we
choose a corresponding spatial nonlocality. For this we have to
choose an \emph{\AAA auxiliary \EEE wave speed} $\ttv $ and arrive at the
\emph{auxiliary problem with spatial nonlocality}: 
\begin{equation}
  \label{eq:Aux1}
   \dot {\wt u}(t,x) = D \wt u_{xx}(t,x) + F(\wt u(t,x)) + \gamma \int_0^\infty\!
   \Gamma(\tau) \, \wt u(t,x{+} \ttv  \tau) \dd \tau. 
\end{equation}
The basis of our theory is the following simple proposition that
connects the existence of traveling waves for the original problem with
that of the auxiliary one.

\begin{proposition}
 \label{pr:TimeSpace}
 If \eqref{eq:Aux1} has a bounded traveling-wave solution
 $\wt u(t,x)=\calU( x{-}\ttc t)$ and the wave speed $\ttc $ matches the
 auxiliary speed $ \ttv $ occurring as parameter in \eqref{eq:Aux1}, then
 $u(t,x)= \calU( x{-}\ttv t) $ is also a solution of \eqref{eq:NaguMemo}.
\end{proposition} 
\begin{proof} Obviously, the partial derivatives and the local function $F(u)$
  in \eqref{eq:Aux1} and  \eqref{eq:NaguMemo} coincide, 
  so it remains to match the integral terms. The straightforward calculation \AAA
\begin{align*}
&\int_0^\infty\!\Gamma(\tau)\, \wt u(t,x{+}\ttv \tau) \dd \tau =
\int_0^\infty\!\Gamma(\tau)\, \calU(x{+}\ttv \tau-\ttc  t ) \dd \tau  \\
& \overset{\ttv  =\ttc }{=} \int_0^\infty\! \Gamma(\tau) \, 
\calU(x + \ttv(t{-}\tau)) \dd \tau = 
\int_0^\infty \! \Gamma(\tau) \,u(t{-}\tau,x) \dd \tau,
\end{align*}  
relying strongly on $\ttc = \ttv$, turns the spatial nonlinearity into a temporal
memory. \EEE This gives the desired result. 
\end{proof}

The above result does not specify the form of the traveling wave, hence it is
applicable to traveling pulse, traveling fronts, or to (quasi-)periodic  wave
trains.

\subsection{Comparison principles for the auxiliary equation}
\label{sec:comparison}

\AAA We are now in a position to apply the general theory for nonlocal
parabolic equations as developed in 
\EEE  \cite{Chen97EUAS,AchKue15TWBE} to our auxiliary problem \eqref{eq:Aux1}
for \AAA $\gamma \in [0,\gamma_*]$. For this, we define the nonlinear operator
$u\mapsto \calA_{\gamma,\ttv }[u] $ via 
\begin{align}
  \label{eq:calA}
\calA_{\gamma,\ttv }[u](x)&:= D u_{xx}(x) + F(u(x))  +  \gamma \int_0^\infty\! \Gamma(\tau)
  \, u(x{+} \ttv \tau ) \dd \tau 
\end{align}
in the notation of \cite{Chen97EUAS}.  We first observe the relation
$\calA_{\gamma,\ttv}[\lambda \bm1] = F_\gamma(\lambda)\bm1$ for all
$\lambda \in \R$, where $\bm1$ denotes the constant function $ u(x) = 1
$. Moreover, the Fr\'echet derivative reads
\[
\big(\calA'_{\gamma,\ttv}[u](\varphi)\big)(x) = D \varphi_{xx}(x)
+F'(u(x))\varphi(x) + \gamma \int_0^\infty \!\Gamma(\tau) \, \varphi(x{+} \ttv
\tau) \dd \tau.
\]
Hence, we obtain
$\big(\calA'_{\gamma,\ttv}[u{+}v](\bm1)\big)(x) -
\big(\calA'_{\gamma,\ttv}[u](\bm1)\big)(x) = F'(u(x){+}v(x))- F'(u(x)) $. With
this and the assumption $F\in \rmC^2(\R)$ from (H1), the two assumptions (A1)
and (A3) in \cite[Sect.\,2]{Chen97EUAS} are satisfied. 

The crucial and nontrivial assumption (A2) in \cite{Chen97EUAS} concerns the
strong comparison principle and is the content of the following
proposition. To formulate it we introduce the notion of super- and
subsolutions for \eqref{eq:Aux1}: We \EEE  call $u^*$ a supersolution and $u_*$ a
subsolution if the relations  $\dot{u}^* - \calA_{\gamma,\ttv }[u^*]\geq 0$ and
$\dot{u}_* -\calA_{\gamma,\ttv }[u_*]\leq 0$ hold, respectively.

\begin{proposition}[Strong comparison principle]
\label{pr:comparison}
Let the Hypotheses (H1)--(H2) and $\gamma \in [0,\gamma_*]$ hold. If $u^*$ is a
supersolution and $u_*$ is a subsolution of \eqref{eq:Aux1} such that
$u^*(0, \cdot) \geq u_*(0, \cdot)$ and \AAA
$u^*(0, \cdot){-} u_*(0, \cdot) \not\equiv 0$, then we have
$u^*(t, \cdot) \gneqq u_*(t, \cdot)$ for all $t> 0$.
\end{proposition}
\begin{proof} \AAA We follow some ideas in the proof of
  \cite[Thm.\,5.1]{Chen97EUAS} for establishing condition (C2). \EEE By
  assumption the difference $w(t,x) := u^*(t,x) - u_*(t,x)$ satisfies
  $w(0,\cdot) \geq 0$ and
\begin{align*}
\dot{w} \geq D w_{xx} + F(u^*) - F(u_*) + \bbJ w 
\quad\text{with}\quad
(\bbJ w)(t,x) := \gamma \int_0^\infty \!\Gamma(\tau) w(t,x{+}\ttv \tau) \dd\tau .
\end{align*}

\emph{Step 1.} We show by contradiction that the inequality
$\dot{w} \geq D w_{xx} + \bbJ w $ implies $w \geq 0$.  We set \AAA
$K := 4\gamma_* {+}4D$ \EEE with $\gamma_*$ from (H1). By assuming $w\ngeq0$,
there exist $\delta>0$ and $T>0$ such that $w(t,x) > - \delta \ee^{Kt}$ for
$t\in \left[0,T \right[$ and $\inf_{x\in\bbR} w(T,x) =
-\delta\ee^{KT}$. Without loss of generality, we may assume
$w(T,0) < -\frac34 \delta \ee^{KT}$.

Next we define the comparison function
$\varphi_\sigma(t,x) := -(\frac12+ \sigma z(x))\, \delta \ee^{Kt}$ \AAA with
$z(x)=\frac{1+3x^2}{1+x^2}$, \EEE where the parameter $\sigma > 0 $ will be
fixed later.  On the one hand, using $z\geq 1$ we have, for
$\sigma \geq \frac12$, the estimate
\begin{align*} 
\ts  \varphi_\sigma(t,x) \leq -\big(\frac12 + \frac12\,1\big) \, \delta \,\ee^{Kt}
  =  -\delta \ee^{Kt} \leq w(t,x)  
\end{align*}
for all $t \in [0,T]$ and $x\in\bbR$. On the other hand, using $z(0)=1$ for
$\sigma = 1/4$ we find 
\begin{align*}
 \ts \varphi_{1/4}(T,0) =-\big(\frac12 + \frac14 1\big) \, \delta \,\ee^{KT}  
  = -\frac34 \delta \ee^{KT} > w(T,0) .
\end{align*}
We define $\sigma_*$ as follows:
\[
\sigma_*:=\inf \Sigma \quad \text{with }
\Sigma:=\bigset{\sigma >0}{ \forall\, (t,x)\in [0,T]\ti \R:\ \varphi_\sigma(t,x)
  \leq w(t,x) } 
\]
and observe that the above estimates imply $\sigma_*\in [\frac14,\frac12]$. By
continuity, we see that $\Sigma$ is closed, hence $\sigma_*\in \Sigma$, which
implies $\sigma_*>\frac14$.  Using $z(x)> 2$ for $|x|> 2$ for all
$\sigma\geq \frac14$, we have
$\varphi_\sigma(t,x) < -(\frac12+\frac14 2) \delta\ee^{K t}\leq w(t,x)$. Hence,
for $\sigma \in {[\frac14,\sigma_*[} \subset \R\setminus \Sigma$ there exists
$(t_\sigma, x_\sigma) \in [0,T] \ti [-1,1]$ such that
$\varphi_\sigma(t_\sigma, x_\sigma) \geq w(t_\sigma, x_\sigma)$.

By the continuity of $w$ and $(\sigma,t,x)\mapsto \varphi_\sigma(t,x)$ and by
compactness of $[0,T] \ti [-1,1]$, we hence
find $(t_*,x_*) \in [0,T] \ti [-1,1]$ such that $\varphi_{\sigma_*}(t_*,x_*) \geq
w(t_*,x_*)$. Obviously, we have $t_*>0$, but $t=T$ may be possible.  
However $\sigma_*\in \Sigma$ means $\varphi_{\sigma_*}(t,x) \leq 
w(t,x)$ for all $(t,x)\in [0,T]\ti \R$. Thus, we conclude $w(t_*,x_*)=
\varphi_{\sigma_*}(t_*,x_*)$, $w_x(t_*,x_*)= \pl_x\varphi_{\sigma_*}(t_*,x_*)$,
\[
w_{xx}(t_*,x_*) \geq \pl^2_x\varphi_{\sigma_*}(t_*,x_*), \quad \text{and }
\left\{ \ba{cl}\dot w(t_*,x_*)= \dot\varphi_{\sigma_*}(t_*,x_*) &\text{for }t_*\in
  {]0,T[}, \\  \dot w(t_*,x_*)\leq \dot\varphi_{\sigma_*}(t_*,x_*) &\text{for }t_*= T. 
\ea\right.
\]

Using (H2) (nonnegativity of $\Gamma$) and $\gamma \in [0,\gamma_*]$ together
with the inequalities (i) $\frac34 \leq \frac12+ \sigma_*z(x_*)$ and 
(ii) $\frac12+\sigma_* z(x)\leq 2$ for all $x \in \R$, 
 the following chain of inequalities for the
particular point $(t_*,x_*)$ holds: 
\begin{align*}
-\tfrac34 \delta K\ee^{Kt_*} 
& \overset{\text{(i)}}\geq -\big(\tfrac12 + \sigma_* z(x_*)\big) \delta K
\ee^{Kt_*} = \dot{\varphi}_{\sigma_*}(t_*,x_*) 
  \\
&\geq \dot{w}(t_*,x_*)
\geq D w_{xx}(t_*,x_*) + \bbJ w(t_*,x_*) \\
& \!\!\! \overset{w\geq \varphi_{\sigma_*}}\geq 
  D \pl_x^2\varphi_{\sigma_*}(t_*,x_*) + \ts \gamma \int_0^\infty \!\Gamma(\tau)
  \, \varphi_{\sigma_*}(t_*,x_*{+}\ttv \tau)) \dd \tau\\
& \overset{\text{(ii)}}\geq  
-D \sigma_* z''(x_*)\delta \ee^{Kt_0} -2 \gamma \delta \ee^{Kt_*} \geq -\big(
2D+2\gamma_*)\delta \ee^{Kt_*} = - \tfrac12 \delta K \ee^{Kt_*},
\end{align*}
where we used $z''(x)\leq 4$ and $\sigma_*\leq \frac12$ and the definition of
$K= 4\gamma_*{+}4D$. \EEE
Thus, we have reached a contradiction and the assertion $w \geq 0$ is proven.

\emph{Step 2.} Without loss of generality, we can assume that $u_*(t,x)$ and
$u^*(t,x)$ only attain values in the bounded interval
$[-\ttC_\mathrm{max}, \ttC_\mathrm{max}]$ for some constant
$\ttC_\mathrm{max} > 0$ depending on the roots of $F$ according to
(H2). Therefore, we have $F(u^*) - F(u_*) \geq K_0 w$ with
$K_0 :=  - \min_{|u|\leq \ttC_\mathrm{max}}F'(u) >0$ and
\begin{align*}
\dot{w} \geq D w_{xx} + F(v^*) - F(v_*) + \bbJ w 
\geq D w_{xx} - K_0 w + \bbJ w  \geq  D  w_{xx} - K_0 w,  
\end{align*}
\AAA where the last estimate follows because of $w\geq 0$ from Step 1 and (H2),
giving $\bbJ w\geq 0$. 

Finally, we use that $w(0,\cdot)$ is nonnegative and not identical to
$0$. Hence, the solution $\psi$ of the linear equation
$\dot\psi=D\psi_{xx} - K_0 \psi$ is strictly positive, as it is given by
$\psi(t,\cdot) = \ee^{-K_0t} H_D(t){*}w(0,\cdot)$, where $H_D(t)$ is the strictly
positive heat kernel.  We now set $W=\ee^{K_0t} (w{-}\psi)$ and obtain
$W(0,\cdot)\equiv 0$ and $\dot W \geq D W_{xx} $. As in Step 1 (with $\bbJ=0$) we
obtain $W\geq 0$ and conclude
\[
u^*(t,x) - u_*(t,x)=w(t,x) \geq \psi(t,x) >0 \quad \text{for all } t>0 \text{
  and }x\in \R. 
\]
This is the desired  strong comparison principle.
\end{proof}

\AAA As an important technical tool we obtain the following simple result
concerning the speed of traveling fronts. 

\begin{proposition}[Comparison of speeds]
\label{pr:CompSpeed} 
Assume that the auxiliary equation \eqref{eq:Aux1} has a traveling front
$u(t,x)=\calU(x{-}ct)$. If there is a traveling-front subsolution
$w_*(t,x)=W_*(x{-}c_*t)$ satisfying
\begin{subequations}
\begin{align}
\label{eq:OrderLimits}
 \lim_{\xi\to -\infty} W_*(\xi)<  \lim_{\xi\to -\infty} \calU(\xi)<
 \lim_{\xi\to +\infty} W_*(\xi) <  \lim_{\xi\to +\infty} \calU(\xi),
\end{align}
then we have $c\leq c_*$. If there is a traveling-front supersolution
$w^*(t,x)=W^*(x{-}c^*t)$   satisfying 
\begin{align}
  \lim_{\xi\to -\infty} \calU(\xi)< \lim_{\xi\to -\infty} W^*(\xi) <  \lim_{\xi\to +\infty} \calU(\xi) <\lim_{\xi\to +\infty} W^*(\xi),
\end{align}
\end{subequations}
then we have $c\geq c^*$.
\end{proposition}
\begin{proof} It suffices to show the result for the subsolution, since the
  proof for the supersolutions is analogous. 

  As the limits of $W_*$ and $\calU$ at $\xi=\pm\infty$ are strictly ordered,
  we can shift $W_*$ to the right to make it smaller for $t=0$. More precisely,
  we set $u^*(t,x)=\calU(x{-}ct)$ and $u_*(t,x)=W_*(x{-}\xi{-}c_*t)$ and find
  $\xi $ big enough such that $u^*(0,x) = \calU(x) >
  W_*(x{-}\xi)=u_*(0,x)$. The comparison principle in Proposition
  \ref{pr:comparison} now implies
\begin{equation}
  \label{eq:CompareTF.TF}
  u^*(t,x)=\calU(x{-}ct)>
W_*(x{-}\xi{-} c_*t)=u_*(t,x) \quad \text{ for all } t\geq 0 \text{ and }x\in \R. 
\end{equation}

Now we assume $c> c_*$, insert $x= \frac12(c{+}c_*)t$ into
\eqref{eq:CompareTF.TF}, and take the limit $t\to \infty$. This leads to 
$\lim_{\xi\to -\infty} \calU(\xi) \geq \lim_{\xi\to +\infty} W_*(\xi)$, which
contradicts the middle assumption in \eqref{eq:OrderLimits}. Hence the
assumption $c>c_*$ is false and $c\leq c_*$ is established. 
\end{proof}
\EEE

\subsection{Existence of traveling fronts for the auxiliary equation}
\label{su:FrontsAux}

\AAA We are now in the position to formulate our result concerning the
existence of traveling fronts in the auxiliary equation \eqref{eq:Aux1}. The
proof will be a direct application of the corresponding result
\cite[Thm.\,5.1]{Chen97EUAS} for spatially nonlocal equations. We obtain a
two-parameter family $\calU_{\gamma,\ttv}$ of traveling fronts depending on the
auxiliary speed $\ttv\in \R$ and the strength $\gamma \in [0,\gamma_*]$ of the
nonlocal term. \EEE

\begin{proposition}
\label{pr:ExiFrontAux}
Let the Hypotheses (H1) and (H2) hold and let $\gamma \in [0,\gamma_*]$. Then, for
 all $\ttv \in \R$ there exists a unique
(up to translation) traveling-front solution
\[
    u(t,x)=\calU_{\gamma,\ttv }\big(x{-} \ttc  t\big)\text{ solving
      \eqref{eq:Aux1} with wave speed } \ttc  = \ttC(\gamma,\ttv ), 
\] 
\AAA which is characterized by 
\begin{subequations}
  \label{eq:TravWave}
\begin{align}
 \label{eq:TravWaveODE}
&{-} \ttc \,\calU'_{\gamma,\ttv}(\xi) = D\,\calU''_{\gamma,\ttv}(\xi) +
F(\calU_{\gamma,\ttv}(\xi)) + \gamma \int_0^\infty \Gamma(\tau) \, 
  \calU_{\gamma,\ttv}(\xi{+}\ttv \tau) \dd \tau,   
\\ \label{eq:TravWave.BC}
& \calU_{\gamma,\ttv }(\xi)\to u_-^\gamma \ \text{ for }\xi\to -\infty \quad
  \text{ and } \quad  
  \calU_{\gamma,\ttv }(\xi)\to u^\gamma_+ \ \text{ for }\xi\to +\infty .
\end{align}
\end{subequations}
These \EEE traveling fronts satisfy the properties
\begin{align}
  \label{eq:WaveProp}
  \calU_{\gamma,\ttv } \in \rmC^2(\bbR), \quad
  \calU_{\gamma,\ttv }'(\xi) > 0 \text{ on } \bbR , \quad
  \calU'_{\gamma,\ttv }(\xi) \to 0 \text{ as } |\xi|\to\infty .
\end{align}
Moreover, they are globally asymptotically stable in
the following sense: There exists $\kappa > 0$ such that for all 
solutions $u$ of \eqref{eq:Aux1}  satisfying
\begin{align}
  \label{eq:initial-value}
  \forall\, x\in\R: u(0,x) \in [u_-^\gamma{-}\kappa, u^\gamma_+{+}\kappa] , 
  \qquad
  \liminf_{x\to+\infty} u(0,x) > u^\gamma_\rmm, \quad
  \limsup_{x\to-\infty} u(0,x) < u^\gamma_\rmm, 
\end{align}
there exist constants $\xi$ and $K$ (depending on $u(0,\cdot)$) such that 
\begin{align}
  \label{eq:shift-decay}
  \Vert u(t,\cdot) - \calU_{\gamma,\ttv }(\,\cdot\, {-} \ttc  t {+} \xi)
  \Vert_{\rmL^\infty(\bbR)} \leq Ke^{-\kappa t} \quad\text{for } t\geq 0 . 
\end{align}
\end{proposition}
\begin{proof}
For the evolution equation $\dot{u} = Du_{xx} + G(u,J{*}S(u))$ in
\cite[Eqn.\,(5.1)]{Chen97EUAS}, we distinguish the cases $\ttv=0$ and $\ttv\neq
0$. In the former case, there is no nonlocal term, and  we set $G(u,p)=F(u) +
\gamma u$ and $J\equiv 0$. In the latter case we \EEE identify the quantities 
\begin{align*}
G(u,p) = F(u) +  \gamma p , \quad S(u) = u , \quad
J=J_v: s \mapsto \frac{1}{|\ttv|} \,\Gamma\big( {-}\frac{s}{\ttv} \big), 
\end{align*}
where $\Gamma(\tau)=0$ for $\tau< 0$ is assumed. 

It remains to verify the Assumptions (D1)--(D4) of \cite[Thm.\,5.1]{Chen97EUAS}:
\begin{itemize}
\itemsep-0.2em
\item[(D1)] Clearly, the operator $\calA_{\gamma,\ttv }$ defined in
  \AAA \eqref{eq:calA} is translation invariant, $\calA_{\gamma,\ttv }
  [u(\,\cdot\,{+}h)](x)  = \calA_{\gamma,\ttv } [u(\cdot)](\,\cdot\,{+}h)$ for
  all $h \in \bbR$. The desired properties of the function $F_\gamma: u \mapsto
   F(u)+\gamma u$ characterized by $\calA_{\gamma,\ttv}(\lambda \bm1)
   =F_\gamma(\lambda)\bm1$ follow directly from (H1). \EEE

\item[(D2)] \AAA The condition (H2) is stronger than our condition because it
  additionally asks $\Gamma \in \rmC^1(\R)$ and $\int_\R |\Gamma'(\tau)|\dd
  \tau < \infty$. However, a close inspection of the proof of
  \cite[Thm.\,5.1]{Chen97EUAS} reveals that these additional conditions are not
  needed in the case $D>0$. Indeed, (D2) is used to derive
  (C2) and (C4) there. However, (C2) is the strong comparison principle, which
  holds according to Proposition \ref{pr:comparison}, while (C4) follows from
  classical parabolic regularity theory because of $D>0$. \EEE 

\item[(D3)] The function $G(u,p)= F(u)+\gamma p$ and $S(u)=u$ are
  smooth with $\partial_p G(u,p) = \gamma\geq 0$ and $S'(u) = 1 > 0$.

\item[(D4)] This holds because of $D>0$. 
\end{itemize}
Therefore, \cite[Thm.\,5.1]{Chen97EUAS} yields the desired existence of
traveling fronts.
\end{proof}

\AAA The comparison principle is not only useful for establishing existence and
uniqueness of traveling fronts. It will also be essential to derive 
qualitative properties of the function $ (\gamma,\ttv ) \mapsto
\ttC(\gamma,\ttv)$. We first derive upper and lower bounds for $\ttC$ and then
its continuity, which will be crucial to construct
traveling fronts for the memory equation \eqref{eq:NaguMemo}. The ideas of the
proof follow \cite[Lem.\,3.2 \& Thm.\,3.5]{Chen97EUAS}, but they are much more
explicit, thus providing realistic bounds by assuming reasonable bounds for
$F_\gamma$. In Figure \ref{fig:FgammaBounds} we display the way in which
$F_\gamma$ needs to be estimated. 

\begin{theorem}[Bounds on the front speed]
\label{thm:SpeedBounds} 
Let the Hypotheses (H1) and (H2) hold and fix $\gamma\in [0,\gamma_*]$ and
$\ttv \in \R$.  Assume further that $F_\gamma$ satisfies the estimates
\begin{equation}
  \label{eq:Ass.Fgamma}
  \begin{aligned}
&F_\gamma(u) \geq -\Phi_* \text{ for } u\in [u^\gamma_-,u^\gamma_\rmm]\quad
\text{and} \quad 
F_\gamma(u) \leq \Phi^* \text{ for } u\in [u^\gamma_\rmm,u^\gamma_+],\\
& \exists\, a_*,b_*\in [u^\gamma_\rmm,u^\gamma_+]\ \exists\, \alpha_*>0: \quad
a_*< b_* \  \text{ and } \ F_\gamma(u) \geq \alpha_*(u{-}a_*) \text{ for } u
\in [a_*,b_*], \\
& \exists\, a^*,b^*\in [u^\gamma_-,u^\gamma_\rmm]\ \exists\, \alpha^*>0: \quad
b^*< a^* \  \text{ and } \ F_\gamma(u) \leq \alpha^*(u{-}a^*) \text{ for } u
\in [b^*,a^*].
\end{aligned}
\end{equation}
Then, the speed $\ttc=\ttC(\gamma,\ttv)$ of the traveling front
$\calU_{\gamma,\ttv}$ satisfies 
\begin{subequations}
  \label{eq:SpeedEstim}
\begin{align}
  \label{eq:SpeedEstim.vpos}
\ttv\geq 0:\quad &  -\max\Big\{  \Big(\frac{\Phi^*D}{b^*{-}a^*} \Big)^{1/2} ,\, 
   \frac{\Phi^*\gamma \wh g_1\ttv}{\alpha^*(b^*{-}a^*)} \Big\}
   - \gamma \wh g_1\ttv  \leq  \ttC(\gamma,\ttv) \leq
\Big(\frac{\Phi_*D}{u^\gamma_+{-}u^\gamma_\rmm} \Big)^{1/2}  ,
\\   \label{eq:SpeedEstim.vneg}
\ttv \leq 0:\quad &- \Big(\frac{\Phi^*D}{u^\gamma_\rmm{-}u^\gamma_-} \Big)^{1/2}
\leq \ttC(\gamma,\ttv) \leq 
  \max\Big\{  \Big(\frac{\Phi_*D}{b_*{-}a_*} \Big)^{1/2} ,\, 
   \frac{\Phi_*\gamma \wh g_1|\ttv|}{\alpha_*(b_*{-}a_*)} \Big\}
   + \gamma \wh g_1|\ttv| 
\end{align}
\end{subequations}
with $\wh g_1=\int_0^\infty \tau \Gamma(\tau) \dd \tau >0$ from (H2). 
\end{theorem} 
\EEE
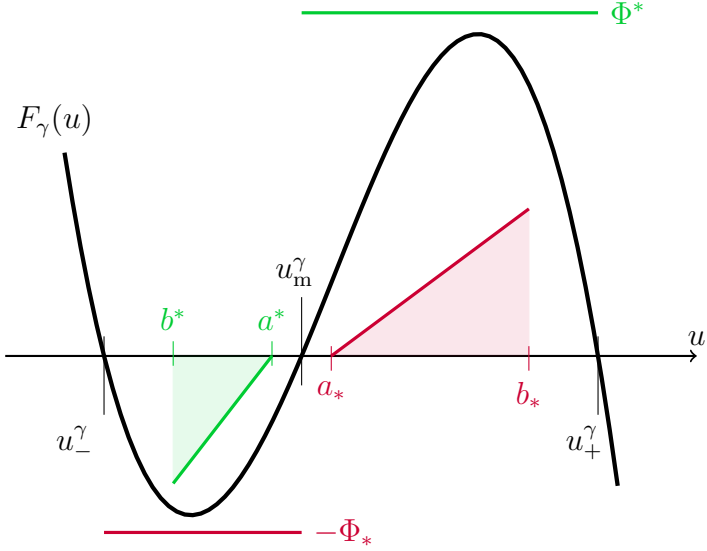
\begin{figure}
\begin{minipage}{0.65\textwidth}
\centering
\begin{tikzpicture}[scale=1.3] 
\draw[thick,->] (-1,0)-- (6,0) node[above]{$u$};
\draw[black, ultra thick, domain=-0.4:5.2, samples =70]
 plot(\x,{-\x*(\x-2)*(\x-5)*0.4});
\node[black] at (-0.5,2.4){$F_\gamma(u)$};
\draw[black] (0,0.2)--(0,-0.6) node[below left]{$u^\gamma_-$};
\draw[black] (2,-0.3)--(2,0.6) node[above]{$u^\gamma_\rmm\;\,$};
\draw[black] (5,0.2)--(5,-0.6) node[below left]{$u^\gamma_+$\!\!\!};
\draw[red!80!blue, very thick] (2.3,0)--(4.3,1.5);
\draw[red!80!blue, fill, opacity=0.1] (2.3,0)--(4.3,1.5)--(4.3,0)--cycle;
\draw[red!80!blue ] (2.3,0.1)--(2.3,-0.15) node [below] {$a_*$};
\draw[red!80!blue ] (4.3,0.1)--(4.3,-0.15) node [below] {$b_*$};
\draw[red!80!blue, very thick] (0,-1.8)--(2,-1.8) node[right]{$-\Phi_*$};
\draw[green!80!blue, very thick] (1.7,0)--(0.7,-1.3);
\draw[green!80!blue, fill, opacity=0.1] (1.7,0)--(0.7,-1.3)--(0.7,0)--cycle;
\draw[green!80!blue ] (1.7,-0.1)--(1.7,0.15) node [above] {$a^*$};
\draw[green!80!blue ] (0.7,-0.1)--(0.7,0.15) node [above] {$b^*$};
\draw[green!80!blue, very thick] (2,3.5)--(5,3.5) node[right]{$\Phi^*$};
\end{tikzpicture}
\end{minipage}
\hfill
\begin{minipage}{0.33\textwidth}
\caption{We show how $F_\gamma$ needs to be estimated from above and below to
  obtain the upper and lower bounds in Theorem \ref{thm:SpeedBounds} for the
  speed of traveling fronts $\calU_{\gamma,\ttv}$ for the nonlocal equation
  \eqref{eq:Aux1}. }
\label{fig:FgammaBounds}
\end{minipage}
\end{figure}
\AAA
\begin{proof} We construct suitable traveling fronts as subsolutions $W_*$ and
supersolutions $W^*$ with speeds $c_*$ and $c^*$, respectively. The comparison
principle for the travel speeds in Proposition \ref{pr:CompSpeed} gives the
desired result $c^*\leq \ttc \leq c_*$.

Choosing a small $\delta>0$ and a positive slope $m>0$, we set
$\xi_*:=(a_*{-}u^\gamma_-)/m>0$ and $\lambda_*: = m/(b_*{-}a_*)$ and define the function
$W_*$ in the specific form
\[
W_*(\xi)=\left\{ \ba{cl} 
  u^\gamma_- -\delta  m& \text{for } \xi\leq -\delta,\\  
  u^\gamma_- +m \xi & \text{for } \xi \in [-\delta,\xi_*], \\
  b_*-(b_*{-}a_*) \ee^{-\lambda_* (\xi-a_*/m)}&\text{for }\xi >\xi_*.   \ea \right.
\]
Actually $W_*$ could be smoothed out in $I_\delta:=[-\frac23\delta,-\frac13\delta]$ such
that it lies in $\rmC^2(\R)$ with  $W''_*(\xi)\geq 0$ for $\xi\in
I_\delta$. This smoothening would not destroy the property of a subsolution. 

A main observation is the monotonicity $W'_*(\xi)\geq 0$ and that $W_*$ and
$\calU_{\gamma,\ttv}$ satisfy the ordering conditions \eqref{eq:OrderLimits}
for the limits at $\xi=\pm \infty$.   
It remains to be shown that there is a speed $c_*$ such that
$u_*(t,x)=W_*(x{-}c_*t)$ is a subsolution. To obtain this, we proceed as
follows (using $\xi=x{-}c_*t$):
\begin{align*}
\dot u_*(\xi) -\calA_{\gamma,\ttv}[u_*](\xi)&= -c_*W'_*(\xi) -DW''_*(\xi)
 - F_\gamma\big( W_*(\xi)\big) - \calL_{\gamma,\ttv}[W_*](\xi) \\
\text{ with }\calL_{\gamma,\ttv}[w](\xi) &:=  \gamma\int_0^\infty \Gamma(\tau) 
 \big( w(\xi{+}\ttv \tau) - w(\xi)  \big) \dd \tau. 
\end{align*}
For $\ttv\geq 0$ we have $\calL_{\gamma,\ttv}[W_*]\geq 0$ because of the
monotonicity of $W_*$ and the nonnegativity of $\Gamma$. Hence, we can drop
$-\calL_{\gamma,\ttv}[W_*]$ when showing that $u_*$ is a subsolution. 

For $\ttv<0$ we estimate $0\leq -\calL_{\gamma,\ttv}[W_*](\xi)$ from above 
by considering two regimes for $\xi$ separately: For $\xi\leq -\delta$ we have
$\calL_{\gamma,\ttv}[W_*](\xi)=0$ since $W_*$ is constant on
${]{-}\infty,-\delta]}$. For $\xi > -\delta$ we use $W'_*(\xi)\in [0,m]$ and
obtain
\[
0\leq -\calL_{\gamma,\ttv}[W_*](\xi) \leq  -\gamma \int_0^\infty \Gamma(\tau)\,
\ttv \tau m \,\dd \tau = g_\ttv m \quad
\text{with } g_\ttv :=  \gamma \wh g_1\max\{-\ttv,0\}  .
\]

We derive the conditions that $m$ and $c_*$ have to satisfy in order to
guarantee that $W_*$ is indeed a subsolutions. For this we estimate $ \dot
u_*(\xi) -\calA_{\gamma,\ttv}[u_*](\xi)$ from above on the separate domains and collect
the corresponding conditions: 
\begin{align*}
\xi < -\delta:\quad &- c_*\cdot 0  - D\cdot 0  - 
 F_\gamma\big( u^\gamma_- {-}\delta m\big) - 0 \ \leq 0 ,
\\
\xi\in {]{-}\delta,\xi_*[}:\quad &- c_*\cdot m - D\cdot 0  + \qquad  \Phi_* \quad  +
m \,g_\ttv \leq 0 ,
\\
\xi> \xi_*: \quad & (b_*{-}a_*)\big( ({-}c_* \lambda_*+\lambda_*^2D + \alpha_*)
\ee^{-\lambda_*( \xi-\xi_*)}
 - \alpha_* \big)  + m\,g_\ttv \leq 0.   
\end{align*}
The first condition is always true because $F_\gamma(u)>0$ for
$u < u^\gamma_-$. For the second term we simply choose
$c_*\geq \Phi_*/m + g_\ttv$, where $m$ is still to be
determined.  Hence, it remains to analyze the third condition. For
$\ttv \geq 0$, the last term vanishes and the terms multiplying $\alpha_*$ are
nonpositive. Hence, it suffices to take $c_* \geq \lambda_* D=
m D/(b_*{-}a_*)$. Together with $ c_*\geq \Phi_*/m$ we can choose $m$ optimally
and find that $c_*= \big( \Phi_*D/(b_*{-}a_*)\big)^{1/2}$ guarantees that $W_*$
is a subsolution. Surprisingly, the result does not depend on the slope
$\alpha_*>0$, hence we may optimize $a_*$ and $b_*$ by pushing them to their
limits $u^\gamma_\rmm$ and $u^\gamma_+$, respectively. Thus, we obtain the
upper estimate for $c$ in \eqref{eq:SpeedEstim.vpos}.

For $\ttv<0$ the term involving $\ttvminus>0$ can only be compensated by
$\alpha_*(b_*{-}a_*)$. Setting $m= \theta \alpha_*(b_*{-}a_*)/g_\ttv$ we can
rewrite the third condition in the form 
\[ \ts
\alpha_*(b_*{-}a_*)\Big( \big(-c_*\frac{\theta}{g_\ttv}
+\frac{\alpha_*\theta^2D} {g_\ttv^2}
+1\big)\,\ee^{-\lambda_*( \xi-\xi_*)} - 1 + \theta\Big) \leq 0.
\]
Together with the second condition it remains to
satisfy 
\[
c_* \geq \max\big\{ \tfrac1\theta\,\tfrac{\Phi_* g_\ttv}{\alpha_*(b_*{-} a_*)} +
g_\ttv \, , \, \theta \,\tfrac{\alpha_* D}{g_\ttv} + g_\ttv \big\} .
\]
For $g_\ttv^2 \Phi_* < \alpha_*^2 (b_*{-}a_*) D$ we find an optimal
$\theta\in {]0,1[}$, whereas otherwise the first term in the maximum dominates
and $\theta=1$ gives the smallest bound for $c_*$.  This establishes the upper
estimate in \eqref{eq:SpeedEstim.vneg}.

To obtain the lower estimates, we construct a supersolution in a completely
analogous fashion. We emphasize that now we need to estimate
$-\calL_{\gamma,\ttv}[W^*] $ from below. Since we still have $(W^*)'(\xi)>0$, we
find $-\calL_{\gamma,\ttv}[W^*] (\xi) \geq 0$ for $\ttv\leq 0$, which is now
the easy case leading to the simple lower bound in \eqref{eq:SpeedEstim.vneg}. 
For $\ttv>0$ we then use $-\calL_{\gamma,\ttv}[W^*](\xi)\geq -
\ttv \gamma\wh g_1 \, m$ and obtain the lower bound in \eqref{eq:SpeedEstim.vpos}.  
\end{proof} 
\EEE

\begin{lemma}
\label{lem:continuity}
\AAA Let the Hypotheses (H1) and (H2) hold. Then, \EEE the function
$\ttC : [0,\gamma_*] \ti \bbR \to \bbR$ is continuous.
\end{lemma}
\begin{proof}
  Consider a sequence $(\gamma_n, \ttv_n)_{n\in\bbN}$ with
  $(\gamma_n, \ttv_n) \to (\gamma,\ttv)$ as $n\to\infty$. According to
  Proposition \ref{pr:ExiFrontAux}, for all $n \in \bbN$ there exists a unique
  traveling front $u_n(t,x) = \calU_n(x {-} \ttc_n t)$ for \eqref{eq:Aux1} with
  $\ttc_n = \ttC(\gamma_n, \ttv_n)$.

  \emph{Step 1: Uniform bounds for the sequence $(\ttc_n,\calU_n)$.} Since
  $(\calU_n)_{n \in \bbN}$ is a traveling front, we have
  $ u^{\gamma_n}_-  < \calU_n(\xi) < u^{\gamma_n}_+ $. \AAA By (H1) and
  the form $F_\gamma(u)=F(u)+\gamma u$, the mappings $\gamma \mapsto
  u^\gamma_\pm$ are uniformly continuous and hence bounded, i.e.\ there exists
  $U_*>0$ such that $-U_* \leq  u^\gamma_-< u^\gamma_+ \leq U_*$ for all $\gamma
  \in [0,\gamma_*]$. Hence, 
$\| \calU_n\|_{\rmL^\infty} \leq U_*$ and there exists $R_*>0$ such that
$|F(u)|\leq R_* $ for all $u\in [-U_*,U_*]$. 

Moreover, we can apply the speed bounds for $\ttc_n=\ttC(\gamma_n,\ttv_n)$ from
Theorem \ref{thm:SpeedBounds} to show that $|\ttc_n|\leq C_0$. Indeed, in a
neighborhood of the limit $(\gamma,\ttv)$ we can choose the estimating
quantities $\Phi_*$, $a_*$, $b_*$, $\Phi^*$, $a^*$, and $b^*$ uniform for 
$(\gamma_n,\ttv_n)$ for all $n\geq n_0$. 

Next we show that $\calU_n$ is also bounded in $\rmH^1_\text{loc}(\R)$. For
this we use equation \eqref{eq:TravWaveODE} in the form 
\begin{equation}
  \label{eq:TF.hn}
  -\ttc_n \calU'_n=D\calU''_n + h_n \quad \text{ with } h_n(\xi)=F(\calU_n(\xi))+
\gamma_n \int_0^\infty \Gamma(\tau)\calU_n(\xi{+}\ttv_n \tau) \dd \tau,
\end{equation}
where now $\| h_n\|_{\rmL^\infty} \leq H_*$ for a suitable constant $H_*$.
Moreover, we define the function $\psi_0 \in \rmC^1(\R)$ with
$\psi_0(x)=(x^2{-}1)^2$ for $|x|\leq 1$ and $\psi_0(x)=0$ otherwise, which
satisfies the estimate $\psi'_0(x)^2 \leq 16 \psi_0(x)$. For arbitrary
$\zeta_* \in \R$ we define the test function
$\psi: \xi \mapsto \psi_0(\xi{-}\zeta_*)$ and test equation \eqref{eq:TF.hn}
with $\psi \calU_n$. Setting $\Xi=[\zeta_*{-}1,\zeta_*{+}1]$ this leads to the
estimate
\begin{align*}
&\Delta_{n,\zeta_*}:= \int_\Xi D|\calU'_n|^2\psi \dd \xi 
  = -\int_\Xi D(\calU''_n\calU_n\psi + \calU_n \calU'_n \psi') \dd \xi \\
& \overset{\text{\eqref{eq:TF.hn}}}= \int_\Xi\big((h_n {+}
     \ttc_n \calU'_n)\calU_n \psi  -  D \calU_n \calU'_n\psi') \dd \xi
\\
&\leq  H_*U_*\|\psi\|_{\rmL^1(\Xi)} \! + \big\|  (D\psi)^{1/2}\calU'_n\big\|_{\rmL^21(\Xi)}
 \big\| \ttc_n (\tfrac\psi D)^{1/2}\calU_n 
       -  (\tfrac D\psi)^{1/2}\psi'\calU_n\big\|_{\rmL^21(\Xi)}
\\
&\leq  \tfrac{16}{15} H_*U_* + \Delta_{n,\zeta_*}^{1/2} \Big(
\big(\tfrac{16}{15 D}\big)^{1/2}C_0U_* + 8D^{1/2} U_* \Big). 
\end{align*}
Since $\zeta_*$ was arbitrary and $D\psi(x)\geq D/2$ for $|x|\leq 1/2$, we find 
the uniform estimate 
\[
\sup_{\zeta_*\in \R} \int_{\zeta_*{-}1/2}^{\zeta_*+1/2} |\calU'_n(\xi)|^2  \dd \xi
\leq \tfrac{64}{15D} H_*U_* + \tfrac{8}D \big( \tfrac{16}{15D} C_0^2 U_*^2 + 64
DU_*^2\big) , 
\]
which together with $\|\calU_n\|_{\rmL^\infty}\leq U_*$ is the desired uniform
bound in $\rmH^1_\text{loc}(\R)$. 

Inserting the last result into \eqref{eq:TF.hn}, we first obtain a uniform bound for 
$\calU_n$ in $\rmH^2_\text{loc}(\R)$, and inserting again we find
$\|\calU''_n\|_{\rmL^\infty} \leq C_2 <\infty$ for all $n\in \N$. 
\EEE

\emph{Step 2: Convergent subsequences and passage to the limit.} We can extract
a subsequence (not relabeled) such that $\ttc_n \to \ttc$ and
$\calU_n \rightharpoonup \calU$ weakly in $\rmH^1_\mathrm{loc}(\bbR)$. \AAA We
also fix the translations in such a way that $\calU_n(0)=u^{\gamma_n}_\rmm \to
u^{\gamma}_\rmm $. Weak convergence in $\rmH^1_\text{loc}(\R)$ implies uniform
convergence on any compact interval $K\Subset \R$. Hence, we have $\calU_n(\xi)
\to \calU(\xi)$ for all $\xi  \in \R$ and $\calU(0)=u^{\gamma}_\rmm$ in
particular. Of course, the uniform bounds for $\calU_n$ imply $\calU\in
\rmC^{1,\text{Lip}}(\R)$ with $\| \calU''\|_{\rmL^\infty} \leq C_2$. Moreover,
  since all $\calU_n$ are nondecreasing, we have $\calU'(\xi)\geq 0$ as well. 
\EEE
   
Testing the equation \eqref{eq:TravWaveODE} with
$\Phi \in \rmC^\infty_\mathrm{c}(\R)$ and integrating over the arbitrarily
chosen compact subset $K \Subset \bbR$ with $\mafo{sppt}\Phi \subset K$ yields
\begin{align*}
  \int_K -\ttc_n \calU_n' \Phi \dd\xi = \int_K  \bigg({-} D\calU_n' \Phi' + \Big(
    F(\calU_n) + \gamma_n\int_0^\infty\!  \Gamma(\tau) 
    \calU_n(\cdot {+}\ttv_n \tau) \dd\tau \Big)  \Phi \bigg) \dd\xi . 
\end{align*}
Exploiting the continuity of $F$ and the locally uniform convergence
$\calU_n \to \calU$, the limit $n\to\infty$ leads to 
\begin{align*}
  \int_K -\ttc \, \calU' \Phi \dd\xi = \int_K \bigg(-D\calU' \Phi' + \Big(
    F(\calU) + \gamma \int_0^\infty \Gamma(\tau)  \calU(\cdot
    {+} \ttv  \tau) \dd\tau \Big) \Phi \dd\xi . 
\end{align*}
Thus, the pair $(\ttc,\calU)$ is a solution of \eqref{eq:TravWaveODE},
which satisfies $\calU(0)=u^\gamma_\rmm$. 

\emph{Step 3: Nontriviality of $\calU$.} We still need to show that $\calU$ is
not equal to the constant solution $u^\gamma_\rmm$. Since we already know that
$\calU$ is monotone, the limits $U_\pm=\lim_{\xi\to \pm \infty}\calU(\xi)$
exist. It is easy to see that these limits satisfy $F_\gamma(U_\pm)=0$. Hence,
it suffices to show $U_- \lneqq u^\gamma_\rmm$ and $U_\gamma \gneqq u^\gamma_\rmm$. For
this, it is sufficient, to find a $\delta>0$ such that $\calU'_n(0)\geq
\delta$, which implies $\calU'(0)\geq \delta >0$. 

To show this, we consider the case $\ttv_n\leq 0$ and estimate $\calU_n$ from
above on ${]{-}\infty,0]}$,  while the case $\ttv_n\geq 0$ is treated
analogously by estimating $\calU_n$ from below on ${[0,\infty[}$. 
Because of $\gamma_n\to \gamma$, we find $f_*>0$ 
such that $F_{\gamma_n}(u) \leq -\frac{f_*}{u^{\gamma_n}_\rmm{-} u^{\gamma_n}_-} 
(u{-}u^{\gamma_n}_-)(u^{\gamma_n}_\rmm-u)\leq 0$ holds for all $u\in [u^{\gamma_n}_-,
u^{\gamma_n}] $ and  $n \in\N$. Hence, using $\calU'_n\geq 0$ and $\ttv_n\leq 0$ 
we obtain 
\begin{align}\nonumber
D\calU''_n&= - \ttc_n \calU'_n -F(\calU'_n) - \gamma_n \int_0^\infty
\Gamma(\tau)\calU_n(\xi{+}\ttv_n \tau) \dd \tau 
\\
\label{eq:PhasePlaneCompare}
&\geq -|\ttc_n|\calU'_n -F_{\gamma_n}(\calU_n) \geq 
-(|\ttc|{+}1)\calU'_n + \frac{f_*}{u^{\gamma_n}_\rmm{-} u^{\gamma_n}_-}
 (\calU_n{-}u^{\gamma_n}_-)\, (u^{\gamma_n}_\rmm{-} \calU_n). 
\end{align} 
With this, we can compare the curve
$\calC_n:\xi\mapsto (\calU_n(\xi),\calU'_n(\xi))$ in the phase plane for
$(U,U')$ with the curve $\wt C_n$ generated by the solution $\wt U_n$ of the
ODE  $ D U'' = -(|\ttc|{+}1)U' +  \frac{f_*}{u^{\gamma_n}_\rmm{-} u^{\gamma_n}_-}
(U{-}u^{\gamma_n}_-)\, (u^{\gamma_n}_\rmm{-} U)$, satisfying $\wt U'_n\geq 0$
and $(\wt U_n(\xi),\wt U'_n(\xi)) \to (u^{\gamma_n}_-,0)$ for $\xi\to -\infty$.

The function $\calU_n$ has the expansion 
$\calU_n(\xi)=u^\gamma_- + d_n\ee^{\lambda_n\xi}(1{+}o(1))$ for $\xi\to
-\infty$ with $d_n>0$, 
where $\lambda_n$ is the positive root of the characteristic equation
$D\lambda^2 +\ttc_n \lambda +F'(u^{\gamma_n}_-)=0$. Hence, the former curve $\calC_n$
leaves the point $(u^{\gamma_n}_-,0)$ to the right with positive slope
  $\lambda_n>0$, i.e.\ $U' =\lambda_n(U{-}u^{\gamma_n}_-)+$h.o.t. 
Similarly, the curve $\wt C_n$ generated by the solution $\wt U_n$ has 
the expansion $U' =\wt \lambda_n(U{-}u^{\gamma_n}_-)+$h.o.t., where
$\wt\lambda_n$ is the positive root of $D\lambda^2 -(|\ttc|{+}1) \lambda -f_*
=0$. Because of $\ttc_n\geq -(|\ttc|{+}1)$ and $F'(u^{\gamma_n}_-)\leq -f_* <0$ we have
$\lambda_n> \wt\lambda_n>0$, which implies that the curve $\calC_n$ lies above
$\wt C_n$ in a neighborhood of $ (u^{\gamma_n}_-,0)$. 

Because of the comparison \eqref{eq:PhasePlaneCompare} we know that $\calC_n$
must stay above $\wt C_n$ until both curves hit the line $U=u^{\gamma_n}_\rmm$.
Thus, choosing $\wt\xi_n$ such that $\wt U_n(\wt\xi_n) = u^{\gamma_n}_\rmm = 
\calU_n(0)$ and $\wt U_n(\xi)< u^{\gamma_n}_\rmm$ for $\xi \leq \wt\xi_n$, we obtain
$\calU'_n(0)\geq \wt U'_n(\xi_n)$, see also 
Figure \ref{fig:PhasePlane}. By simple scaling we see that $\wt U_n$ has the form
\begin{align*}
&\wt U_n(\xi)= u^{\gamma_n}_- + \big(u^{\gamma_n}_\rmm {-}
u^{\gamma_n}_-\big)\,Y(\kappa (\xi{-}\xi_n)) \ \text{ where } Y''+Y'- \sigma Y(1{-}Y)=0
, \\
&\kappa=\frac{|\ttc|{+}1}D, \quad \sigma=
\frac{f_* \,D}{(|\ttc|{+}1)^2}, \quad
Y(-\infty)=0,\quad Y(0)=1, \quad Y'(\xi)>0. 
\end{align*}
With this, we arrive at the desired result
\[
  \calU'_n(0) \geq \wt U'_n(\xi_n)= \big(u^{\gamma_n}_\rmm {-}
u^{\gamma_n}_-\big)\, \kappa\, Y'(0) \geq \frac12 \big(u^{\gamma}_\rmm {-}
u^{\gamma}_-\big)\,\kappa\, Y'(0) =:\delta >0.
\]

In summary we have shown that the limit pair $(\ttc,\calU)$ satisfies the ODE
\eqref{eq:TravWaveODE} as well as the boundary conditions
\eqref{eq:TravWave.BC}. Hence, by the uniqueness in Proposition
\ref{pr:ExiFrontAux} we conclude $\ttc = \ttC(\gamma,\ttv)$, and the desired
continuity of $\ttC$ follows from $(\gamma_n,\ttv_n)\to (\gamma,\ttv)$ and
$\ttc_n = \ttC(\gamma_n,\ttv_n) \to \ttc = \ttC(\gamma,\ttv)$. Indeed, the
convergence along the chosen subsequence converts into convergence of the full
sequence, because the limit of any convergence subsequence of
$(\ttc_n,\calU_n)$ is uniquely determined. 
\end{proof}

\begin{figure}
\begin{tikzpicture}
\node at (1.8,-0.4) {$u^{\gamma_n}_-$};
\node at (5.6,-0.4) {$u^{\gamma_n}_\rmm$};
\draw[->] (0,1)--(8.1,1) node[right] {$U$};
\draw[->] (0.1,-0.4)--(0.1,4.2) node[right] {$U'$};
\node[red, above] at (4.3,3.7) {$\calC_n$};
\node at (4,2) {\includegraphics[width=0.5\textwidth]{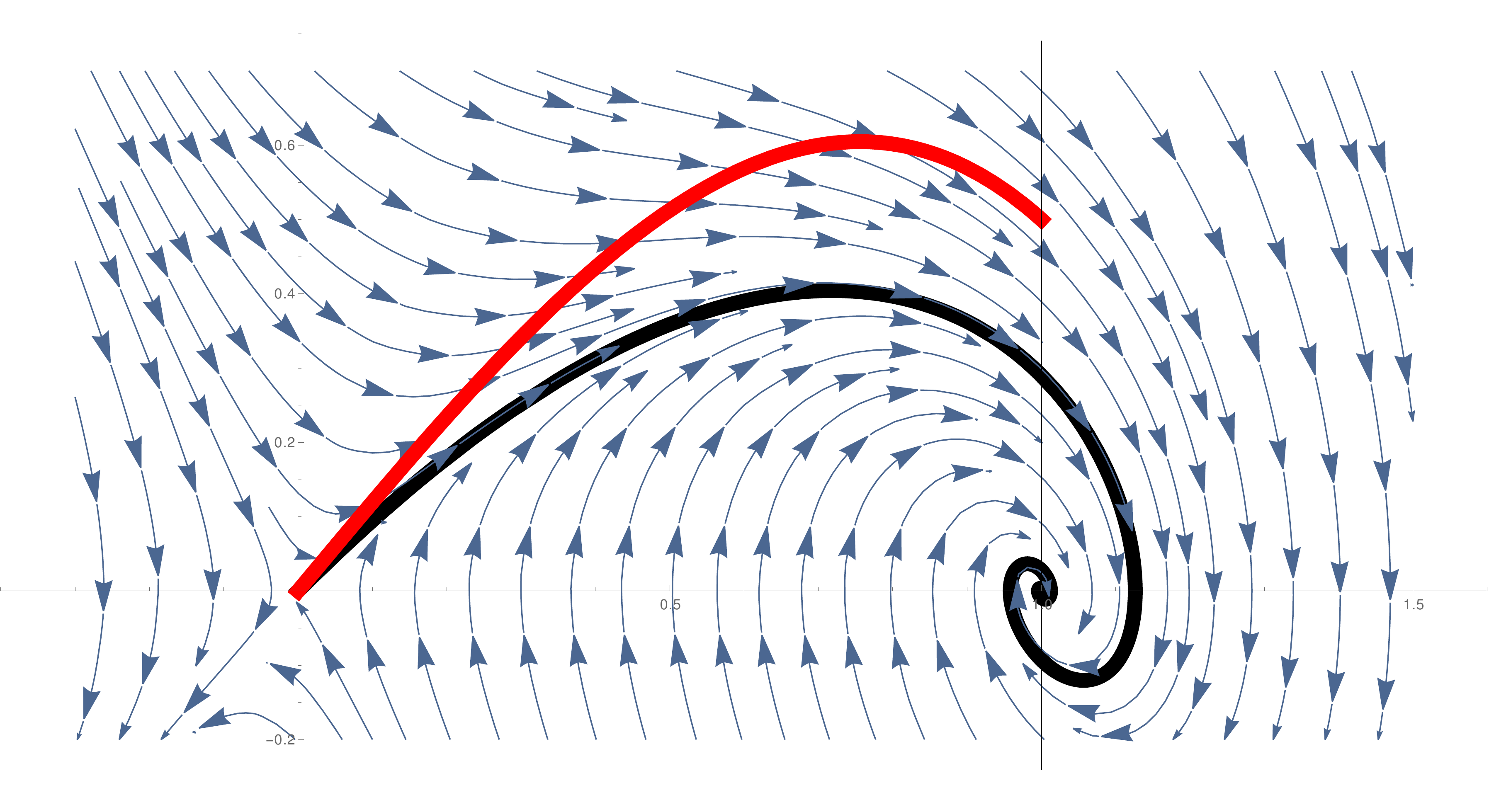}};
\end{tikzpicture}
\hfill
\begin{minipage}[b]{0.42\textwidth}
\caption{Phase plane for $(U,U')$ displaying the stream plot of the ODE 
$U''+\ttc \,U'+ f_* (U{-}u^{\gamma_n}_-) ( 
u^{\gamma_n}_\rmm {-}U)=0$ together with the {\color{red}(red)}
curve\newline
$\calC_n: \xi \mapsto (\calU_n(\xi),\calU'_n(\xi))$ \newline
lying above the unstable manifold of $(u^{\gamma_n}_-,0)$ 
(black curve).}
\label{fig:PhasePlane}  
\end{minipage}
\end{figure}
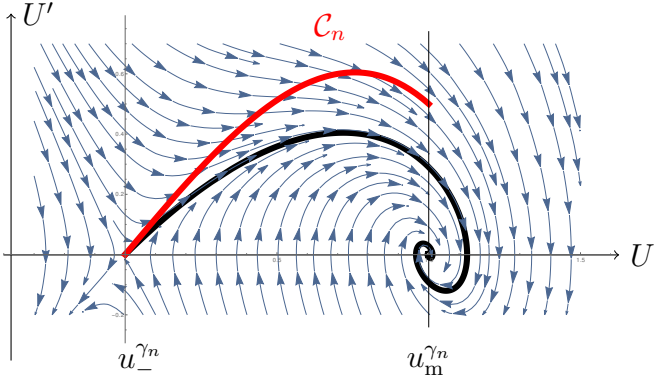

\subsection{Traveling fronts for the memory equation}
\label{su:Results}

According to Proposition \ref{pr:TimeSpace}, we obtain a traveling front for
the memory equation \eqref{eq:NaguMemo} by finding $\ttv \in \R$ such that
$\ttv =\ttC(\gamma,\ttv )$. To guarantee existence of such solutions, we derive
suitable bounds and monotonicity properties for the wave speed function $\ttC$,
where again the comparison principle for sub- and supersolutions is useful, see
Propositions \ref{pr:comparison} and \ref{pr:CompSpeed}. \AAA We recall our
choice that all traveling fronts $\calU_{\gamma,\ttv}$ satisfy
$\calU_{\gamma,\ttv}(-\infty)=u^\gamma_- \lneqq u^\gamma_+ = 
\calU_{\gamma,\ttv}(\infty)$.\EEE

\begin{lemma}[Monotonicity of $\ttC$] 
\label{lem:compare_beta-v} 
For all $\gamma, \wt\gamma \in [0,\gamma_*]$ and all $\ttv,\wt\ttv  \in \bbR$, we
have the following implications: 
\begin{align}
\label{eq:C.diff.v}
& (a) \qquad\qquad \ttv \leq \wt\ttv   & \hspace*{-5.5em} \Longrightarrow \quad
\ttC(\gamma,\ttv) &\geq \ttC(\gamma,\wt\ttv) ;
\\
  \label{eq:C.diff.beta}
  & (b) \ \gamma \leq \wt\gamma \text{ and } 
 \max\{ u^\gamma_-,u^{\wt\gamma}_-\}\geq 0  & \hspace*{-5.5em} \Longrightarrow
  \quad \ttC(\gamma,\ttv )&\geq  \ttC(\wt\gamma , \ttv ) .
\end{align}
\end{lemma}
\begin{proof} \AAA Throughout the proof we set
  $u_{\gamma,\ttv}(t,x)=\calU \big(x{-}\ttC(\gamma, \ttv)t\big)$ and insert one
  traveling front into the parabolic equation for the other front, thus
  obtaining a super or a subsolution. Then, Proposition \ref{pr:CompSpeed}
  provides a comparison of the wave speeds.\EEE

\emph{Ad (a).} We insert $u_{\gamma,\wt\ttv}$ into the equation for
$u_{\gamma,\ttv}$ and obtain
\begin{equation}
\label{eq:relation.beta}
\dot u_{\gamma,\wt\ttv} {-} \calA_{\gamma,\ttv}[u_{\gamma,\wt\ttv}] = \AAA \underbrace{\dot
u_{\gamma,\wt\ttv} {-}\calA_{\gamma,\wt\ttv}[u_{\gamma,\wt\ttv}]}_{=0} + \EEE \gamma 
\int_0^\infty \!\Gamma(\tau)\;\!\big(u_{\gamma,\wt\ttv}(t,x{+}\wt\ttv\tau) {-}
u_{\gamma,\wt\ttv}(t,x{+}\ttv\tau)\big) \dd \tau.
\end{equation}
Using $\Gamma(\tau)\geq 0$, $\calU'_{\gamma,\wt\ttv}(\xi)\geq
0$, and $\ttv \leq \wt\ttv$, we see that the integral is nonnegative. Together
with $\gamma\geq 0$, we find $\dot u_{\gamma,\wt\ttv}
-\calA_{\gamma,\ttv}[u_{\gamma,\wt\ttv}]\geq 0$, i.e.\ $u_{\gamma,\wt\ttv}$ is
a supersolution of $\dot u = \calA_{\gamma,\ttv}[u]$, whereas $u_{\gamma,\ttv}$
is a (sub)solution. Hence, Proposition \ref{pr:CompSpeed} gives \eqref{eq:C.diff.v}.

\emph{Ad (b).} \AAA Assume $u^\gamma_-\geq 0$ such that $u_{\gamma,\ttv}(t,x)\geq
u^\gamma_-\geq 0$. \EEE We insert this solution into the equation for
$u_{\wt\gamma,\ttv}$  and obtain
\begin{equation}
\label{eq:relation.ttv}
\dot u_{\gamma,\ttv} {-} \calA_{\wt\gamma,\ttv}[u_{\gamma,\ttv}] = \underbrace{\dot
u_{\gamma,\ttv} {-}\calA_{\gamma,\ttv}[u_{\gamma,\ttv}]}_{=0} + \EEE
(\gamma{-}\wt\gamma)  
\int_0^\infty \!\Gamma(\tau)\;\! u_{\gamma,\ttv}(t,x{+}\ttv\tau) \dd \tau.
\end{equation}
Since the integral is nonnegative, the assumption $\gamma\leq \wt\gamma$
implies that $u_{\gamma,\ttv}$ is a subsolution whereas $u_{\gamma,\wt\ttv}$
is a (super)solution. Hence,  \eqref{eq:C.diff.beta} follows via Proposition
\ref{pr:CompSpeed}. 

If $u^\gamma_-< 0$ we can rely on $u^{\gamma}_-\geq 0$. Interchanging the role
of $\gamma$ and $\wt\gamma$ in \eqref{eq:relation.ttv}, we see that
$u_{\gamma,\wt\ttv}$ is a supersolution of $\dot u = \calA_{\gamma,\ttv}[u]$
and  \eqref{eq:C.diff.v} follows again. 
\end{proof}

Having collected all preliminary estimates, we can prove the existence of
traveling fronts for the original \AAA memory equation
\eqref{eq:NaguMemo}. Moreover, we are able to give bounds on the corresponding
wave speed. \EEE

\begin{theorem}[Main result]
\label{th:MainResult}
  Let the Hypotheses (H1) and (H2) hold. Then, for all $\gamma \in [0,\gamma_*]$ the
  equation $\ttc = \ttC(\gamma,\ttc)$ has a unique solution $\ttc_\gamma$, and
  thus there exists a  unique (up to translation), asymptotically stable
  traveling-front solution $u(t,x) = \ol\calU_\gamma(x{+}\ttc_\gamma t)$ for the
  memory equation \eqref{eq:NaguMemo}, where $\ol\calU_\gamma$  satisfies the
  properties \eqref{eq:TravWave} and \eqref{eq:WaveProp}.  
\end{theorem} 
\begin{proof} 
  For a fixed $\gamma \in [0,\gamma_*]$ the function
  $\ttv \mapsto \ttC(\gamma,\ttv )$ is continuous and non-increasing according
  to Lemma \ref{lem:continuity} and relation \eqref{eq:C.diff.v}, respectively.
  Thus, there is a unique solution to $\ttv = \ttC(\gamma,\ttv )$, which we call
  $\ttc_\gamma$. According to Proposition \ref{pr:ExiFrontAux} there is a unique
  (up to translation) traveling front
  \AAA $\ol\calU_\gamma:=\calU_{\gamma,\ttc_\gamma}$ for the auxiliary equation 
  \eqref{eq:Aux1}with $\ttv=\ttc_\gamma$.  

  Because of $\ttv=\ttc_\gamma= \ttC(\gamma,\ttc_\gamma)$ we can return to the
  memory equation \EEE using Proposition \ref{pr:TimeSpace}, and the result is
  established.
\end{proof}

Next we combine the available estimates and monotonicity of the wave speed
$\ttC(\gamma, \ttv)$ to give bounds on $\ttc_\gamma$. \AAA Before doing so, we
recall the classical result that for the local equation, i.e.\ $\ttv=0$, the
sign of the wave speed $\ttC(\gamma,0)$ is opposite to the sign of
\[
\calF(\gamma):=\int_{u^{\gamma}_-}^{u^\gamma_+} F_\gamma(u) \dd u. 
\]
Indeed, to see this, we multiply \eqref{eq:TravWaveODE} by $\calU'$ and
integrate over $\R$ leading to 
\[
- \ttC(\gamma,0) \int_\R  \calU'_{\gamma,0}(\xi)^2 \dd \xi= \lim_{R\to \infty} 
\Big(\frac D2 \calU'_{\gamma,0}(\xi)^2 \Big)\Big|_{-R}^R + \int_\R
F_\gamma(\calU_{\gamma,0}(\xi))\calU'_{\gamma,0}(\xi)\dd\xi=  \calF(\gamma). 
\]
In particular, for $\calF(\gamma)=0$, we obtain $\ttC(\gamma,0)=0$, which means
that a standing wave exists. Obviously this implies $\ttc_0=0$, i.e.\ also the
memory equation \eqref{eq:NaguMemo} has the same standing wave solution as the
local equation with nonlinearity $F_\gamma$. 
\EEE

\begin{corollary}[Bounds on the wave speed]
  \label{co:SpeedBounds} Let the assumptions (H1) and (H2) hold and let
  $F_\gamma$ satisfy \eqref{eq:Ass.Fgamma}, then for all
  $\gamma \in [0,\gamma_*]$, we have the following bounds for the speeds
  $\ttc_\gamma$ of the traveling front $\ol \\calU_\gamma$ for the memory
  equation \eqref{eq:NaguMemo}: 
\begin{equation}
  \label{eq:cgamma.estim}
\begin{aligned}
\calF(\gamma)\leq 0 & \quad \Longrightarrow \quad 0 \leq \ttc_\gamma \leq
\ttC(\gamma,0) \AAA \leq \Big( \frac{\Phi_* D}{u^\gamma_+{-}u^\gamma_\rmm}\Big)^{1/2} ,
\\
\calF(\gamma) \geq 0 & \quad \Longrightarrow \quad  0 \geq \ttc_\gamma \geq 
\ttC(\gamma,0) \AAA \geq - \Big( \frac{\Phi^* D}{u^\gamma_\rmm {-}u^\gamma_-}\Big)^{1/2} .
\end{aligned}
\end{equation}
\end{corollary}
\begin{proof} \AAA Since $ \ttv \mapsto \ttC(\gamma, \ttv)$ is non-increasing
  and continuous, we easily see that the solution $\ttc_\gamma$ of
  $\ttv=\ttC(\gamma,\ttv)$ always lies between $0$ (not included) and
  $\ttC(\gamma,0)$ (possibly included). This provides the first two estimates in
  both lines of \eqref{eq:cgamma.estim}. \AAA The last estimate (in both lines) 
  is a direct consequence of \eqref{eq:SpeedEstim} with $\ttv=0$. \EEE
\end{proof}


\section{Derivation of memory equations from local PDEs}
\label{se:Derivation}

Equations with memory can be derived from local models if so-called internal
variables $w$ are eliminated.  Changes in the main variable $u$ induce
instantaneous changes in $w$, but the internal dynamics of the $w$ system 
leads to the delayed back-coupling from $w$. This effect stays local in space,
if the diffusion in the variable $w$ can be neglected. 

\AAA
\subsection{Parabolic equation coupled to linear ODEs}   
\label{su:ParaODEs}

For the modeling of pulse propagation in nerves, one often uses the coupling 
between a parabolic equation and ODEs, see e.g.\ \cite[Eqn.\,(0.1)]{Carp77GASP} for the
Hodgkin-Huxley equation and \cite{NaArYo62APTL,Deng91EIMT} for the FitzHugh--Nagumo-like
equations, which corresponds to the case $m_\rmw=1$ in the following model:
\begin{equation}
  \label{eq:PDE-ODE-System}
  \dot u = Du_{xx} + F(u) + \sum_{i=1}^{m_\rmw} a_i w_i, \qquad \dot w_i = -
  \lambda_i w_i + b_i u \ \text{ for }i=1,\ldots,m_\rmw,
\end{equation}
where $a_i, \ b_i$, and $\lambda_i$ are fixed real parameters with
$\lambda_i>0$. Clearly, the linear ODEs can be solved locally in space by
$w_i(t,x)=w_i(0,x)\ee^{-\lambda_it} + \int_0^t
\ee^{-\lambda_i(t{-}\tau)} b_i u(\tau,x)\dd \tau$. Assuming that there is
infinite history (which is compatible with our search for traveling fronts that
exist for all time) we may also write $w_i(t,x)=\int_0^\infty
\ee^{-\lambda_i \tau} b_i u(t{-}\tau,x)\dd \tau$. With this we obtain the
memory equation \eqref{eq:NaguMemo} with the specific kernel 
\begin{equation}
  \label{eq:Gamma.Sum}
  \Gamma(\tau)= \frac1\gamma \sum_{i=1}^{m_\rmw} \ee^{-\lambda_i \tau} a_i b_i  \quad
\text{ with  }\gamma = \sum_{i=1}^{m_\rmw}  \frac{a_ib_i}{\lambda_i} . 
\end{equation}
Clearly, our theory is applicable also in cases where some of the
products $a_ib_i$ are negative, as long as $\Gamma$ is nonnegative, e.g.\
$\ee^{-2\tau}- 2\ee^{-3\tau} +
\ee^{-4\tau}=(\ee^{-\tau}{-}\ee^{-2\tau})^2$. Thus, we need a positive feedback
$a_ib_i$ for several species $i$, but we may also allow for a negative feedback
$a_jb_j$ for some components if they are not too big. 

For systems of the type \eqref{eq:PDE-ODE-System} the existence of \EEE traveling
waves, in particular pulses, is studied in \cite{GurRei18PFNS} in detail, even
in cases where  Hypothesis (H2) is violated, i.e.\ $\Gamma$ may change sign.  
The latter is necessary to handle 
non-monotone fronts and pulses. Therefore, the assumptions in
\cite{GurRei18PFNS} are more general with respect to the types of traveling
waves (e.g.\ pulses, non-monotone fronts) under consideration.
	
However, when studying only traveling fronts, the assumptions of the present
article are much more general, since (H2) does not require any assumptions on
the (exponential) structure of $\Gamma$. In particular, we
believe that our approach may be generalized to nonlinear coupling terms as in
Section \ref{su:Generalizations}, whereas the approach in
\cite{GurRei18PFNS} relies on the linear structure of the 
equations for $w_i$. Moreover, the present approach allows to calculate
bounds on the wave speed, \AAA see Theorem \ref{thm:SpeedBounds} and Corollary
\ref{co:SpeedBounds}. \EEE

\subsection{The two-scale homogenization model}
\label{s:TwoscaleHom}

\AAA A significant body of work (cf.\ \cite{HudZin95ETWR, Xin00FPHM,
  Hein01WSRD, BerHam02FPPE, MaScUe07EATW, ChGuWu08TWDP, BodMat14EHTW})
considers the propagation in periodic media. If the period of the oscillating
coefficient is very small one can perform a homogenization and consider
traveling waves in the homogenized system. However, often the diffusion
coefficients of some of the species are very small as well, which leads to a \EEE
coupled system of the form
\begin{equation} 
\label{eq:SystOrig}
\left.\begin{aligned}
\dot v&= \div\big( \bbD_\rmv(\tfrac1\eps x)\nabla v\big) + f_\rmv(\tfrac1\eps x,
v,w)\\
\dot w&=\div\big(\eps^2 \bbD_\rmw(\tfrac 1\eps x) \nabla w\big) +
f_\rmw(\tfrac1\eps x, v,w)
\end{aligned} \right\}  \quad t>0,\ x\in \Omega.
\end{equation}
Here the functions $\bbD_\rmv$, $\bbD_\rmw$, $f_\rmv(\cdot,v,w) \in \R^{m_\rmv}$, and
$f_\rmw(\cdot, v,w)\in \R^{m_\rmw}$ are assumed to be 1-periodic in each component of
$y=\frac1\eps x\in \R^d$. 

It is shown in \cite{MiReTh14TSHN} that under suitable conditions on the diffusion
matrices $\bbD_\rmv$ and $\bbD_\rmw$ and the reaction terms $f_\rmv$ and
$f_\rmw$ the solutions $(v_\eps,w_\eps)$ to the initial value problem converge
in the limit $\eps \to 0 $ to solutions $(V,W)$ of the following two-scale
model. For this we denote by $\bbT^d=\R^d/\Z^d$ the $d$-dimensional torus
obtained by identifying the opposite sides of the unit cube. While the fast
diffusion of the variable $v_\eps$ guarantees that the limit $V$ only depends
on the macroscopic variable $x \in \Omega$, the limit $W$ of the solutions
$w_\eps$ \AAA is a two-scale function depending \EEE also on the microscopic
variable $y\in \bbT^d$: 
\begin{subequations} 
\label{eq:Syst}
\begin{align}
\label{eq:Syst.a}
\dot V(t,x)&= \div_x\big( \bbD_\rmv^\text{eff}\nabla_x V(t,x)\big) +
\int_{\bbT^d} f_\rmv(y, V(t,x),W(t,x,y)) \dd y &&\text{in }\Omega,\\
\label{eq:Syst.b}
\dot W(t,x,y)&=\div_y\big(\bbD_\rmw(y) \nabla_y W(t,x,y)\big) +
f_\rmw(y, V(t,x),W(t,x,y)) && \text{in }\Omega \ti \bbT^d.
\end{align} 
\end{subequations}
Here $\bbD_\rmv^\text{eff}$ is the effective diffusion tensor
obtained by classical homogenization, see e.g.\ \cite{BeLiPa78AAPS}. 
The main point in this theory is that it is not possible to replace
the slowly diffusing component $w^\eps$ by its macroscopic average. We
rather need to keep track of the microscopic distribution of the
$w^\eps$ relative to the underlying periodic microstructure. This is
exactly done by the function $W$ depending on $x$ and $y$. 

\AAA The original theory in \cite{MiReTh14TSHN} and \cite{Reic15TSHS} was
developed for bounded Lipschitz domains $\Omega \subset \R^d$. We show in
Appendix \ref{sec:appendix-TS} how the result can be generalized to equations
posed on the full space $\Omega=\R^d$, which is needed to treat traveling
waves. For this one introduces \EEE the weighted Lebesgue spaces
\begin{align}
  \label{eq:WeightedL2}
  \rmL^2_\varrho(\bbR^d) 
  & := \Bigset{ u \in \rmL^2_\mathrm{loc}(\bbR^d) }
      {\Vert u \Vert^2_{\rmL^2_\varrho(\bbR^d)} := 
      \int_{\bbR^d} \varrho(x) |u(x)|^2 \dd x  < \infty } , 
\\ \nonumber
  \rmL^2_\varrho(\bbR^d\ti \bbT^d) 
  & := \Bigset{  U \in \rmL^2_\mathrm{loc} (\bbR^d\ti\bbT^d) }
   {\Vert U \Vert_{\rmL^2_\varrho(\bbR^d\times\bbT^d)}^2 := 
   \int_{\bbR^d\times\bbT^d} \varrho(x) |U(x,y)|^2 \dd x \dd y < \infty },
\end{align}
\AAA where for a radius $R>0$ we set $\varrho(x)=1/\cosh(|x|/R)$. With this and
the natural conditions on $\bbD_\rmv$, $\bbD_\rmw$, $f_\rmv$, and $f_\rmw$ we
derive the following result in Appendix \ref{sec:appendix-TS}. \EEE

\begin{theorem}[Two-scale homogenization]
\label{thm:ts-conv-2}
Let $(v_\eps, w_\eps)_{\eps>0}$ denote a sequence of solutions to the initial
value problem \eqref{eq:SystOrig} \AAA on $\Omega=\R^d$ with initial conditions
$(v_\eps^0,w_\eps^0)=(v_\eps(0,\cdot), w_\eps(0,\cdot)$ that are bounded in 
$\rmL^\infty(\R^d)^{m_\rmv+m_\rmw}$ and converge as follows: \EEE
\begin{align*}
  v_\eps^0 \to V^0 \text{ in } \rmL^2_\varrho(\bbR^d)
  \quad\text{and}\quad
  w_\eps^0 \AAA \stc \EEE  W^0 \text{ in } \rmL^2_\varrho(\bbR^d\ti \bbT^d) ,
\end{align*}
then the solution $(v_\eps, w_\eps)$ converges to the solution $(V,W)$ of the
two-scale system \eqref{eq:Syst},
\begin{align*}
  \begin{array}{ll}
    v_\eps \rightharpoonup V \ \text{ in } \rmL^2(0,T; \rmH^1_\varrho(\bbR^d)), 
    & \ \nabla v_\eps \wtc \nabla V {+} \nabla_y \wt{V} \text{ in } \rmL^2(0,T;
    \rmL^2_\varrho( \bbR^d\ti \bbT^d)) , \\ 
    w_\eps \stc W \text{ in } \rmL^2(0,T; \rmL^2_\varrho( \bbR^d\times\bbT^d))  , 
    & \eps \nabla w_\eps \wtc \nabla_y W \qquad \text{ in } \rmL^2(0,T;
    \rmL^2_\varrho( \bbR^d\ti\bbT^d)) ,  
  \end{array}
\end{align*}
where $\wt{V} \in \rmL^2([0,T] \times\Omega ; \rmH^1(\bbT^d))$ with
$\int_{\bbT^d} \wt{V} \dd y = 0$ denotes the corrector function.
\end{theorem}

\AAA 
The major advantage of the two-scale model \eqref{eq:Syst} is that it is again
homogeneous in the macroscopic spatial variable $x\in \R^d$, while the periodic
structure is restricted to the microscopic variable $y\in \bbT^d$. Thus, we
have a coupling that is local in $x\in \R^d$ from $V(s,x)\in \R^{m_\rmv}$ to 
$W(s,x,\cdot)\in \rmH^1(\bbT^d)^{m_\rmw} $. At later times $t>s$ the internal parabolic
evolution of $W(t,x,\cdot)$ via \eqref{eq:Syst.b} leads to a delayed 
feedback of $V(s,x)$ for all $s<t$ of memory type.

In particular, it is possible to look for exact traveling
waves for \eqref{eq:Syst} in the form 
\[
(V(t,x),W(t,x,y))= (\calV(x{-}\ttc t), \calW(x{-}\ttc t,y)) \in
\R^{m_\rmv+m_\rmw}.
\] 
Transforming these two-scale solutions back into the one-scale form (called
folding in \cite{MieTim07TSHE,MiReTh14TSHN}) one obtains periodically
oscillating traveling 
waves, namely  
\[
(\wt v_\eps(t,x), \wt w_\eps(t,x))= (\calV(x{-}\ttc t), \calW(x{-}\ttc
t,\frac1\eps x))
\]
that provide the correct first-order approximation of the true solutions
$(v_\eps,w_\eps)$, see Figure \ref{fig:eps} for a plot of such solutions.
\EEE

\subsection{Reduction to a scalar equation}
\label{su:ReductScalar}

\AAA Here we give an explicit example for deriving a scalar memory equation
from a two-scale system with $(v,w) \in \R^1\ti \R^1$. We start by looking at a
specific case for \EEE the full time-dependent system \eqref{eq:Syst}, namely
the FitzHugh-Nagumo case where \AAA $V$ and $W $ are scalar and are defined on
the real line $\Omega = \bbR$. Moreover, we assume linear couplings with
$f_\rmv(y,V,W) = \Phi(y,V) + \alpha(y)W$ and
$f_\rmw(y,V,W) = - b(y) W + \beta(y)V$: \EEE
\begin{subequations}
  \label{eq:NagumoDelay}
\begin{align}
 \label{eq:NagumoDelay.a}
 \dot V(t,x)\quad &=  D_\rmv V_{xx}(t,x) +\int_{\bbT} \Phi(y,V(t,x) )\dd y 
 + \int_{\bbT} \alpha (y)  W(t,x,y)\dd y,\\ 
 \label{eq:NagumoDelay.b}
 \dot W(t,x,y) & = \left( \bbD_w(y) W_{y}(t,x,y) \right)_y - b(y)W(t,x,y) 
   + \beta (y)V(t,x) ,
\end{align}
\end{subequations}
\AAA where $\bbT =\bbT^1=\R_{\!/\Z}$. All functions $\Phi$, $\alpha$, $\beta$
and $b$ are assumed to be continuous. In general, \EEE the coupling parameters
$\alpha(y)$ and $\beta(y)$ may change sign, while the microscopic diffusion
coefficient $\bbD_\rmw(y)$ and the damping factor $b$ are assumed to be
strictly positive. \AAA The solution $\psi(t,y)$ of the linear equation
$\dot \psi=(\bbD(y)\psi_y)_y - b(y)\psi$, $\psi(0,\cdot)=\psi^0$ has the
semigroup representation
$\psi(t,y)=\int_\bbT H(t,y,\wt y) \psi^0(\wt y)\dd \wt y$, where the Greens
function $H(t,y,\wt y)$ is strictly positive for $t>0$ by the maximum principle
for linear parabolic equations, see e.g.\ \cite[Thm.\,12,\,p.\,376]{Evan98PDE}.

By introducing the effective nonlinearity $F(V):= \int_\bbT \Phi(y,V)\dd y$ and
expressing $W$ as a linear functional over the history of $V$ via
\eqref{eq:NagumoDelay.b}, \EEE we obtain a Nagumo equation with memory kernel:
\begin{equation}
  \label{eq:Nag**Memo}
\begin{aligned}
&  \dot V(t,x) = V_{xx}(t,x) + F(V(t,x)) + \int_0^{\infty}\!
  \wh\Gamma(\tau)
  V(t{-}\tau,x)\dd \tau \\
&\text{with } \wh\Gamma(\tau)= \int_{\bbT} \int_{\bbT} H(\tau,y,\wt y) \alpha(\wt y)
\beta(y) \dd y \dd \wt y . 
\end{aligned} 
\end{equation}
\AAA Thus, using $H> 0$ 
a sufficient condition for $\wh\Gamma(\tau)\geq 0$ in our Hypothesis (H2)
is given by $\alpha(y), \beta(y)\geq 0$ for all $y \in \bbT$ or
vice versa $\alpha(y),\beta(y)\leq 0$. 

\begin{remark}[On the positivity of $\wh\Gamma$]
\label{re:Pos.whGamma}
The given conditions on $\alpha$ and $\beta$ are far from optimal. 
Indeed, writing $\bbL\psi=-(\bbD_\rmw\psi_y)_y + b\psi$ we find a complete
orthonormal set $(\psi_n)_{n\in \N}$ in $\rmL^2(\bbT)$ of injunctions, i.e.\
$\bbL\psi_n =\lambda_n \psi_n$ with $0<\lambda_1\leq \lambda_2 \leq \cdots
\leq \lambda_n \to \infty$. 

Expanding the coupling coefficients $\alpha(y)=\sum_\N a_n\psi_n(y)$ and 
$\beta(y)=\sum_\N b_n\psi_n(y)$, we find $
\wh\Gamma(\tau)=\sum_{n=1}^\infty  \ee^{-\lambda_n \tau} a_n b_n$, and conclude
that the property  $a_nb_n\geq 0$ for all $n\in \N$ is sufficient. Hence,
setting $\alpha(y)=\beta(y)=\psi_n(y)$ and noticing that, by the Sturm-Liouville
property, the function $\psi_n$ has $ \lfloor \frac{n-1}2 \rfloor$ zeros,  we
have constructed an example for $\alpha$ and $\beta$ that change sign and still
satisfy $\wh\Gamma(\tau)\geq 0$. 

However, it is not even necessary that all of the products $a_nb_n$ are
nonnegative, see the arguments after \eqref{eq:Gamma.Sum}.
In \cite{GurRei18PFNS} the explicit choices $\alpha(y)= - \sum_{n=1}^{m_*}
\psi_n(y)$ and $\beta(y)=\sum_{n=1}^{m_*} \sigma_n \psi_n(y)$ were made, and
different signs for $\sigma_n$ are explicitly allowed. 
\end{remark}
\EEE

\subsection{A homogenization example}
\label{subsec:HomogExamp}

\AAA In the spirit of Remark \ref{re:Pos.whGamma}  we consider the two-scale \EEE
system  \eqref{eq:NagumoDelay} \AAA in a specific example fulfilling all the
assumptions of our theory. \EEE For the operator $\bbL W :=  -W_{yy} + W$ we
have the eigenfunctions $\psi_{2n}(y) = \sqrt2 \sin(2\pi n y)$ with eigenvalues
$\lambda_{2n} =1{+}(2\pi n)^2$.  We set
\begin{equation*}
	\alpha(y) = \psi_2(y) + \psi_4(y)
	\quad\text{and}\quad
	\beta(y) =  \psi_2(y) + 10 \, \psi_4(y) - \psi_6(y) .
\end{equation*}
Hence, Hypothesis (H2) is satisfied with
$\gamma=\int_0^\infty \wh\Gamma(\tau)\dd \tau = 1/\lambda_2 + 10/\lambda_4
\approx 0.08763> 0$.  Notice that the coupling coefficients $\alpha(\cdot)$ and
$\beta(\cdot)$ do not change sign simultaneously. For the cubic function, we
choose $F(u) = -u(u{-}0.25)(u{-}1)$. For this choice of parameters, the
two-scale solution $(V,W)$ is depicted in Figure \ref{fig:limit}, where the
microscopic average of the component $W$ vanishes, since the constant
eigenfunction $\psi_1(y)=1$ is not activated because of
$\langle \psi_1,\beta\rangle =0$. However, the periodic oscillations of $W$ on
the micro-scale are captured by the two-scale limit.
\begin{figure}[h]
\includegraphics[width=0.5\textwidth]{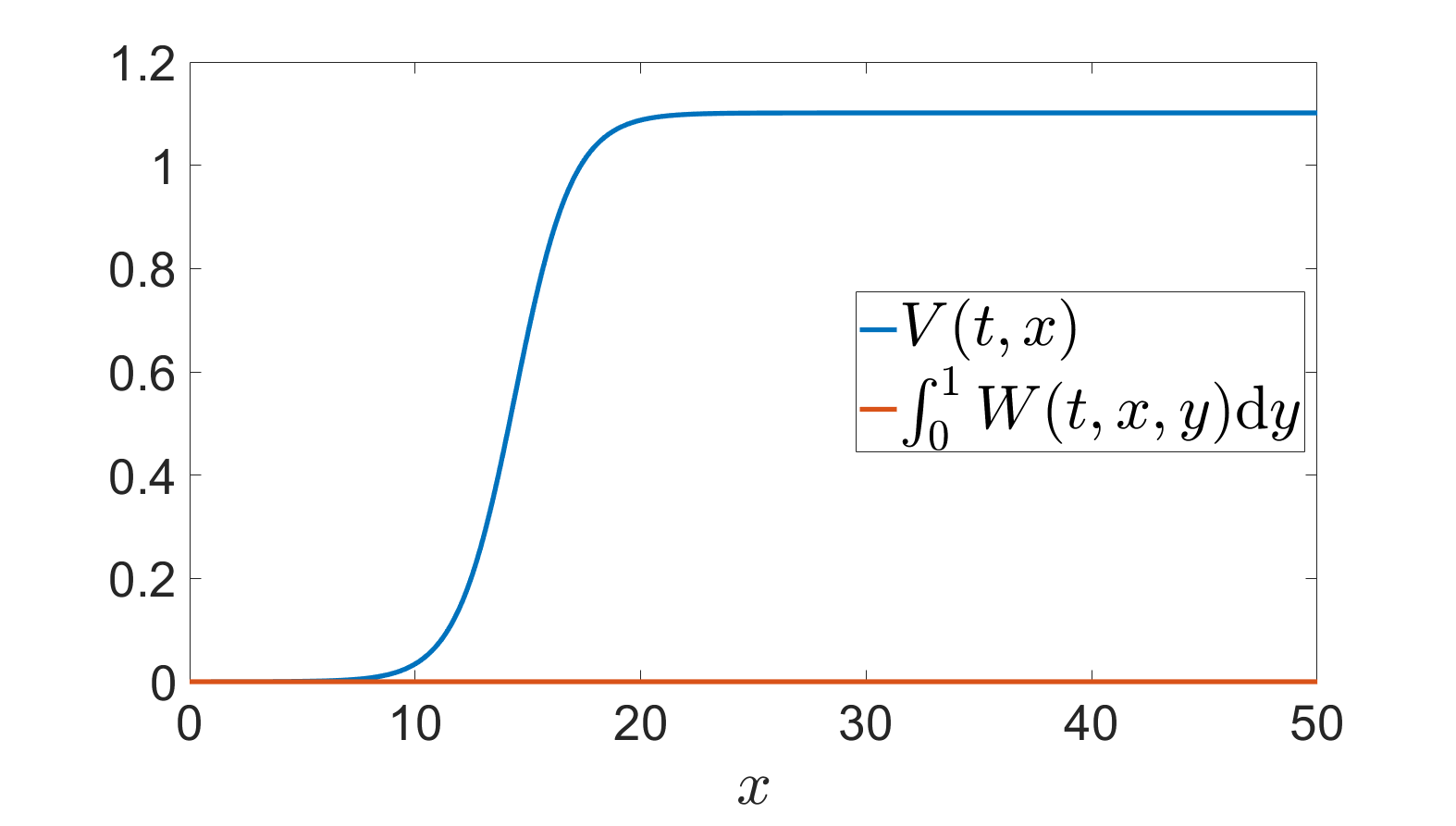}
\includegraphics[width=0.5\textwidth]{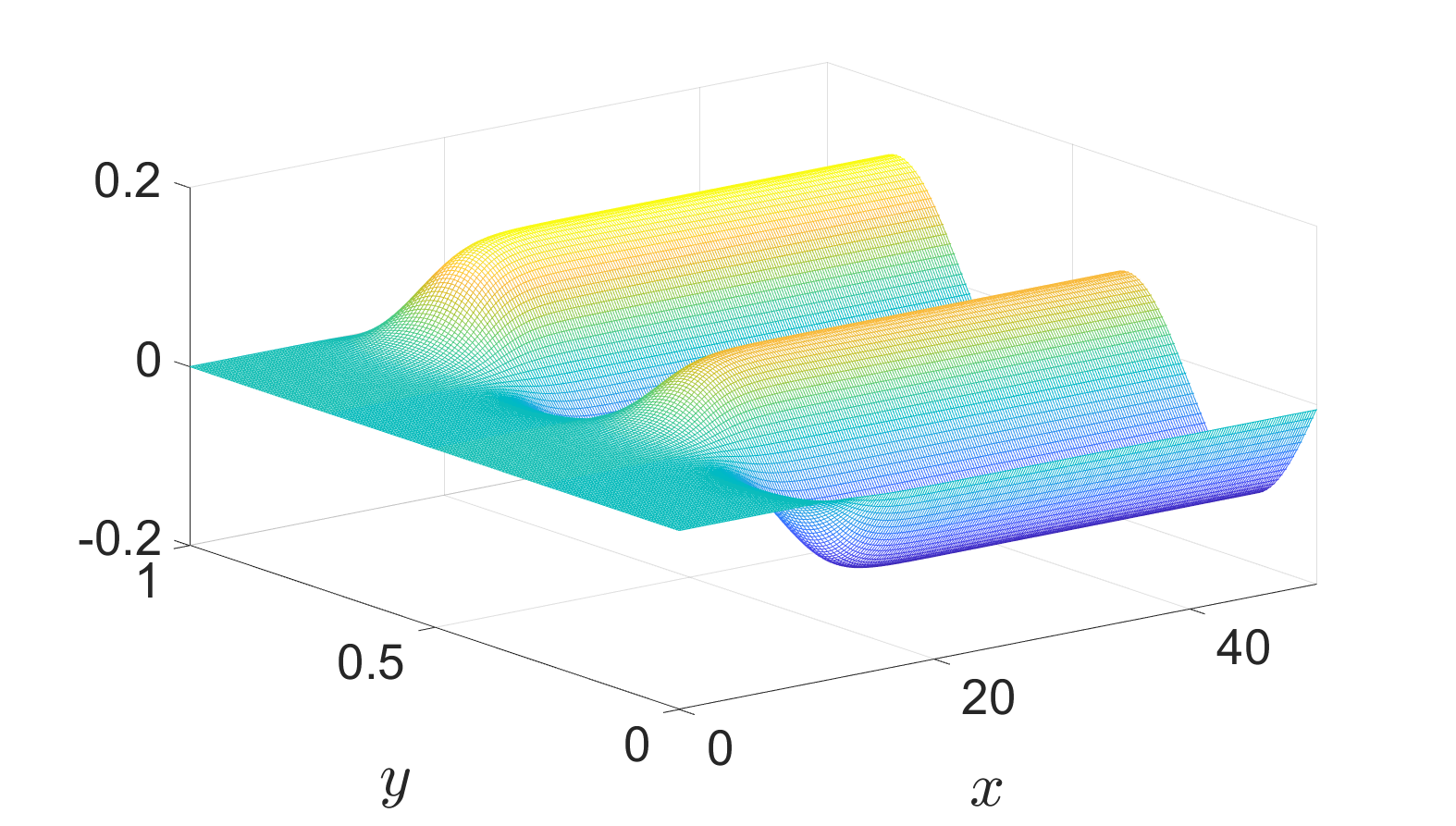}
\caption{The front $(V, W)$ of the two-scale system
  \eqref{eq:NagumoDelay} moves from right to left in $x$. Left: $V$ and
  average $\int_\bbT W(t,x,y) \dd y$. Right: the $W(t,x,y)$-component in
  $(x,y)$-plane.}
\label{fig:limit}
\end{figure}

We can compare the solution $(V,W)$ of our two-scale limit system with the
solution $(v_\eps, w_\eps)$ of the original system \eqref{eq:Syst} with rapidly
oscillating coefficients. In Figure \ref{fig:eps} one can observe that the
$\eps$-periodic coupling coefficients induce $\eps$-periodic oscillations of
the solutions. Whereas the amplitude of the oscillations of the component
$w_\eps$ is of order $O(1)$, the postive diffusion $D_\rmv$ reduces the
amplitude of the oscillations of the component $v_\eps$ to order $O(\eps)$,
such that it vanishes in the limit $\eps\to0$. Notice that the
component $w_\eps$ also changes sign. Overall, the effective behavior of the
oscillating solution $(v_\eps,w_\eps)$ is nicely captured by the two-scale
limit $(V,W)$.
\begin{figure}[h]
\includegraphics[width=0.5\textwidth]{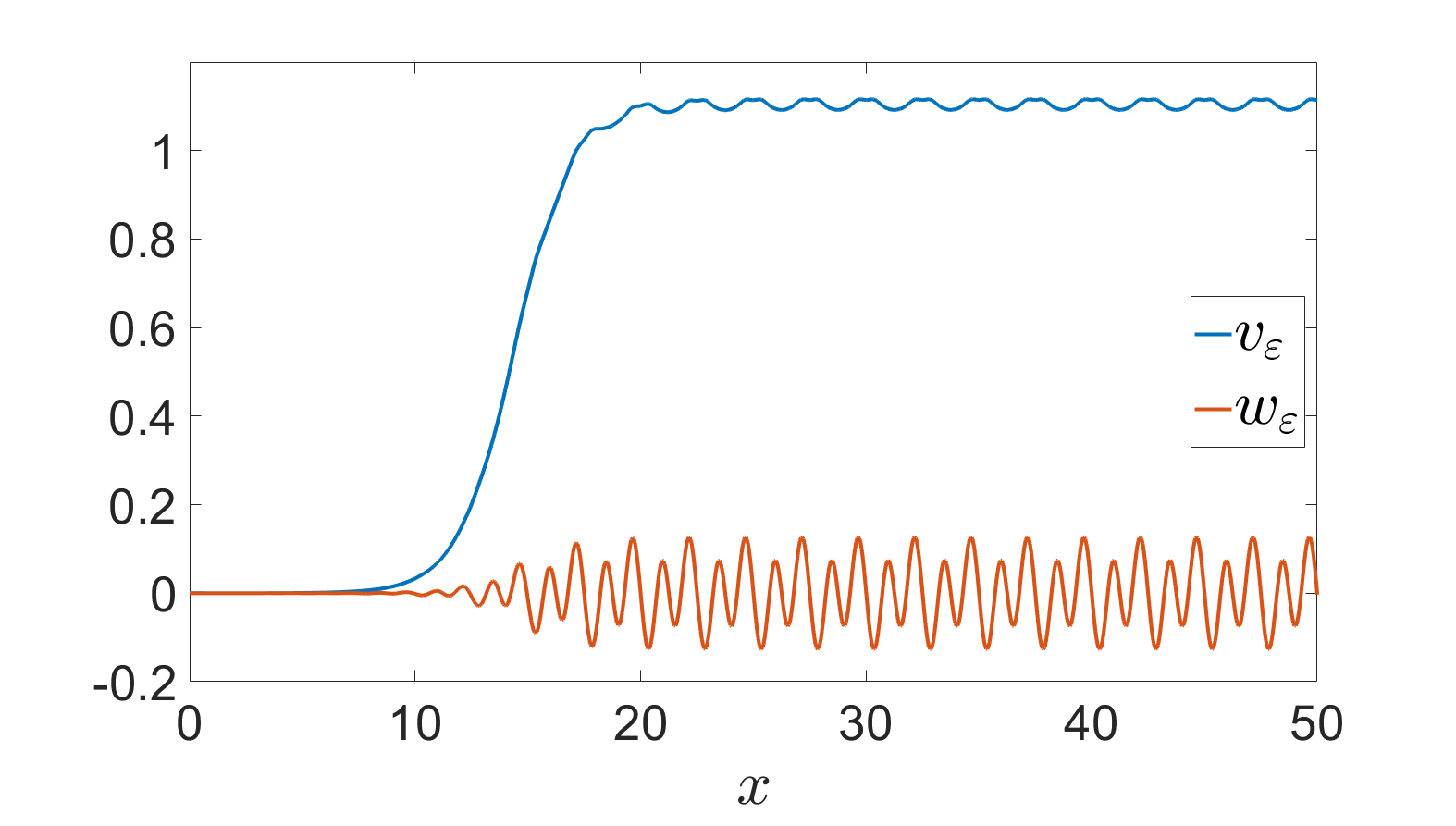}
\includegraphics[width=0.5\textwidth]{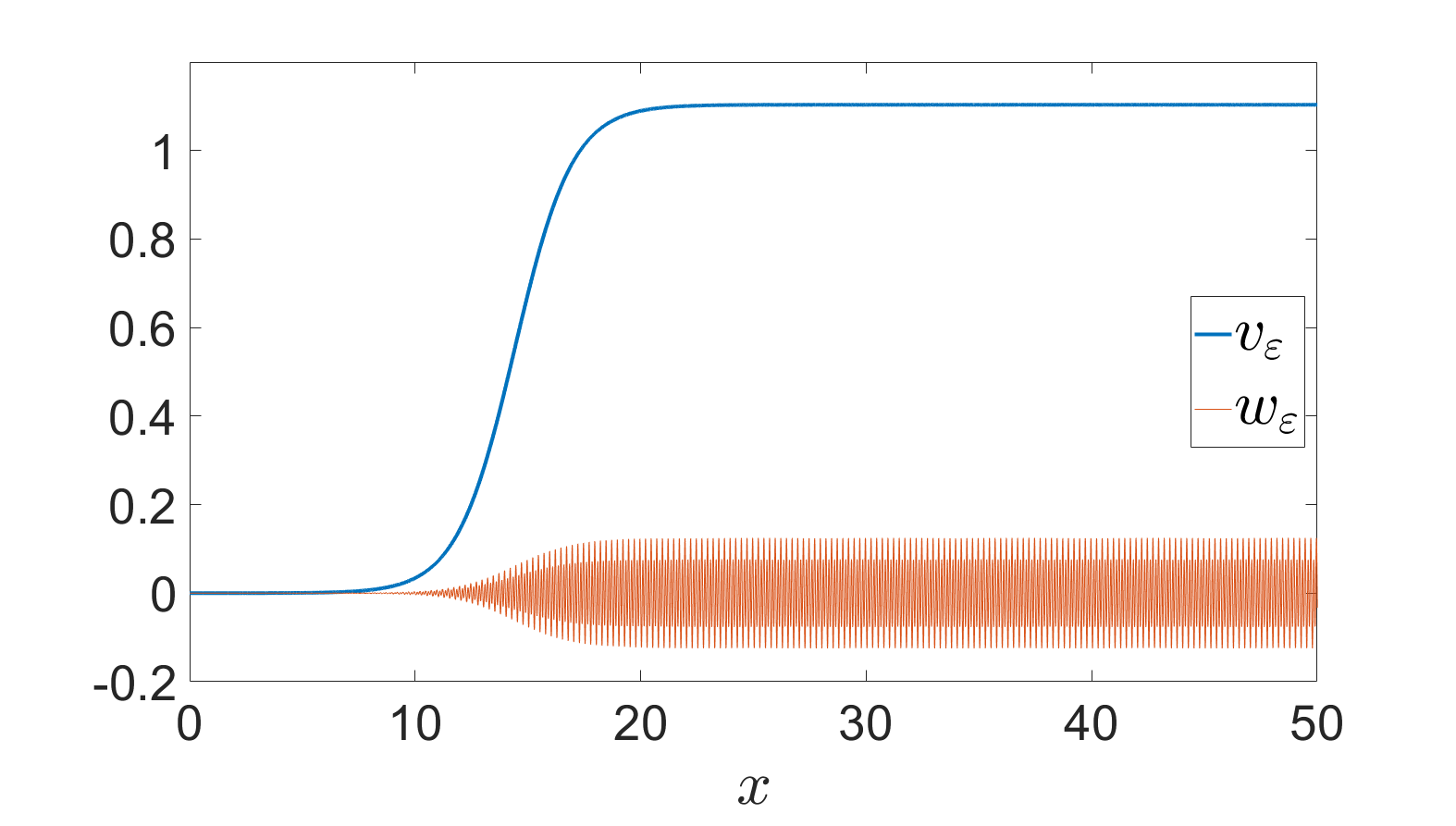}
\caption{Solutions $(v_\eps, w_\eps)$ of the original system \eqref{eq:Syst}:
  left $\eps {=} 2.5$, right $\eps {=} 0.25$.} 
\label{fig:eps}
\end{figure}


\AAA 
\subsection{Possible generalizations}
\label{su:Generalizations}

For mathematical conciseness we have decided to restrict ourselves to memory terms with
a simple linear structure. However, as our theory is many based on the quite
general work in \cite{Chen97EUAS} it is clear that the results can be
generalized in several ways. 

First, the integral memory can be replaced by discrete time delays in the form 
such that the equation reads 
\begin{equation}
  \label{eq:Delay}
  \dot u(t,x) = D u_{xx}(t,x) + F(u(t,x)) + \sum_{j=1}^{j_*} \gamma_j u(t{-}\tau_j,x).
\end{equation}
Assuming $\gamma_j>0$ and setting $\gamma=\sum_{j=1}^{j_*} \gamma_j$, we obtain
the memory equation \eqref{eq:NaguMemo} with a $\Gamma(\tau)=\sum_{j=1}^{j_*}
\frac{\gamma_j}{\gamma} \,  \delta_{\tau_j}(\tau)$. We refer to \cite{WuZou01TWFR}
for a general approach to traveling waves in reaction-diffusion equations
involving linear and nonlinear delay terms.\medskip

Secondly, further generalizations can be obtained by looking at nonlinear
couplings to ODEs, thus generalizing the coupled system
\eqref{eq:PDE-ODE-System}:
\begin{equation}
  \label{eq:Nonl.PDE-ODE}
  \dot u= Du_{xx} +F(u)+ \sum_{i=1}^{m_\rmw} A_i(u,w_i), \qquad
\dot w_i = - \lambda_i w_i + B_i(u),
\end{equation}
where the coupling functions $A_i:\R^2\to \R$ and $B_i:\R\to \R$ are assumed to be smooth
functions with derivatives satisfying $\pl_u A(u,w_i)\leq 0$, $\pl_{w_i}
A_i(u,w_i)\geq 0$, and $B'_i(u)\geq 0$. Again we can express $w_i(t,x)$ as memory
functional depending on  $u(s,x)$ for $s \in {]{-}\infty,t[}$. Hence,
\eqref{eq:Nonl.PDE-ODE}  can be written in the form 
\begin{align*}
 &\dot u(t,x)= Du_{xx}(t,x) +F(u(t,x))+ \sum_{i=1}^{m_\rmw} A_i\big(u(t,x),
 \bbGamma_i{*}B_i(u(\cdot,x))(t) \big), 
\\
&\text{where } \bbGamma_i{*}h(t)= 
 \int_0^\infty \ee^{-\lambda_i \tau} h(t{-}\tau) \dd \tau. 
\end{align*}
Hence we have a memory in time, which is local in the space variable
$x\in \R$.  As in Section \ref{su:aux-prob} we can introduce an auxiliary wave
speed $\ttv$ and turn the memory terms into spatially nonlocal terms. Because
of our assumptions on $A_i$ and $B_i$ we are exactly in the setting of
\cite[Eqn.\,(1.14)]{Chen97EUAS}. Thus, it is expected that the methods
developed here can be extended to such nonlinear couplings.

Thirdly, it seems possible to generalize the theory to handle multidimensional
traveling waves in a cylindrical domain $\Omega=\R\ti\Sigma$ as in
\cite{Gard86EMTW}, but now with memory terms. 
\begin{align*}
\dot u(t,x,y)&= D\big(u_{xx}(t,x,y)+\Delta_y u(t,x,y)\big)\\
& \quad +F(y,u(t,x,y))+ 
\gamma \int_0^\infty \Gamma (\tau) u(t{-}\tau,x,y) \dd \tau 
             &&\text{for } t>0,\ (x,y)\in \R\ti \Sigma, \\
0&= u(t,x,y) &&\text{for }t>0,\ (x,y)\in \R\ti \pl\Sigma. 
\end{align*}
However, then the ``bistability'' Hypothesis (H1) has to be formulated in  terms of the
elliptic problem $u \mapsto D\Delta_y u +F(u)+ \gamma u$ in $\Sigma$, which
should have exactly three solutions $u^\gamma_\alpha \in \rmH^1_0(\Sigma)$ with
appropriate properties. 
\EEE

\section{Applications to the cubic case} 
\label{su:CubicCase}

In this section we study the classical FitzHugh-Nagumo system with a bistable
cubic nonlinearity, namely  
\begin{equation}
  \label{eq:SimpleFHN.A}
  \dot u = u_{xx} +F(u) -\beta v, \qquad \dot v = -v+u ,\quad 
  F(u)=-u(u{-}a)(u{-}1),
\end{equation}
where $a\in {]0,1[}$. Concerning traveling fronts, this system is equivalent to
the memory equation 
\begin{equation}
  \label{eq:CubicMemo}
   \dot u(t,x) = u_{xx}(t,x)  +F(u(t,x) ) -\beta \int_0^\infty \ee^{-\tau}
   u(t{-}\tau,x) \dd \tau . 
\end{equation}
In particular, we see that the previous parameter $\gamma$ is given by
$\gamma=-\beta$,  such that our theory developed above only applies for
$\beta\leq 0$. However,  in the numerical simulations documented below, we are
free to choose $\beta>0$ as well. 

We first observe that the nonlinearity 
\[
F_{-\beta}(u)=-u(u{-}a)(u{-}1) - \beta u
\]
has a bistable structure if and only if $\beta < (1{-}a)^2/4$ and then
\[
u^{-\beta}_-=0,\quad 
u^{-\beta}_\rmm = \frac{1{+}a}2 - \frac12\sqrt{(1{-}a)^2-4\beta}, \quad 
u^{-\beta}_+ = \frac{1{+}a}2 + \frac12\sqrt{(1{-}a)^2-4\beta}.
\]
In \cite{Mcke70NE}, the wave speed of the local model $\dot
u=u_{xx}+F_{-\beta}(u)$ is calculated explicitly,
which gives 
\[
\ttC^\FHN(-\beta,0):= 
 \frac1{\sqrt2}\big( 2u^{-\beta}_\rmm- u^{-\beta}_+\big) =
\frac1{2\sqrt2} \Big( 1{+}a  -
3\sqrt{(1{-}a)^2-4\beta}\: \Big).
\]
In the case $\beta=0$ there is no coupling between the ODE an the PDE, hence
$\ttc_0=\ttC^\FHN(0,0)= (2a{-}1)/{\sqrt2}$ is known.  

In the case $a\in {]\frac12,1[}$, which we fix from now on, the function
$\beta\mapsto \ttC({-}\beta,0)$ changes sign at
\[
\beta_0(a) =  \tfrac19 \big( 2a^2{-}5a {+} 2 \big) \in {]{-}\tfrac19 ,0[}.
\]
With $\ttC^\FHN(0,0)=(2a{-}1)/{\sqrt2}>0$, we obtain $\ttC^\FHN(-\beta,0)>0$ for 
$\beta \in {]\beta_0(a),\frac14(1{-}a)^2]}$ and $\ttC^\FHN(-\beta,0)<0$ for $\beta <
\beta_0(a)$. 

From this, we see that we can apply our existence theory in Theorem
\ref{th:MainResult}  and obtain a unique traveling front for the FitzHugh-Nagumo 
\eqref{eq:SimpleFHN.A}, or equivalently for the memory equation
\eqref{eq:CubicMemo} for all $\beta\leq 0$. This front
connects the values $u^{-\beta}_-=0$  and $u^{-\beta}_+$, and using Corollary
\ref{co:SpeedBounds} it travels with the speed $\ttc_{-\beta}$ satisfying the
bound 
\begin{align*}
0\leq \ttc_{-\beta} \leq \ttC^\FHN({-}\beta,0) & \quad 
 \text{for } \beta \in[ \beta_0(a),0],
\\ 
\ttC^\FHN({-}\beta,0) \leq  \ttc_{-\beta} \leq  0& \quad 
 \text{for } \beta \leq  \beta_0(a).
\end{align*}

For numerical simulations we choose $a=0.6$ determine  $\ttc_{-\beta}$ on the
whole bistable regime $\beta \leq (1{-}a)^2/4=0.04$. For this case, we have
$\ttc_0=\ttC^\FHN(0,0)=0.2/\sqrt2\approx 0.1414$ and 
\[
\ttC^\FHN(-\beta,0)=
\frac1{2\sqrt2}\big(1.6-3\sqrt{0.4^2-4\beta}\:\big) 
\approx 0.566-2.12\sqrt{0.04{-}\beta}
\]
\AAA as upper or lower bound for $\ttc_{-\beta}$. Moreover, we know that
$\ttC^\FHN({-}\beta,0)$ changes sign at  $\beta_0(0.6)= -0.28/9\approx
-0.0311$.

To determine the speed of the traveling front numerically, we solve the coupled
system \eqref{eq:SimpleFHN.A} starting from front-like initial data with the
correct plateau $(u,v)=(u^{-\beta}_+,u^{-\beta}_+)$ for large $x$ and
$(u,v)=(0,0)$ for small $x$. The solution rapidly stabilizes into the front and
the speed $\ttc_{-\beta}$ can be measured. \EEE The numerical simulations
nicely confirm the bounds on the wave speed $\ttc_{-\beta}$ in Corollary
\ref{co:SpeedBounds}. For all $\beta < \beta_0(0.6) \approx -0.0311$, Figure
\ref{fig:speed} shows $\ttC(-\beta, 0) \leq \ttc_{-\beta} \leq 0$. \AAA
\begin{figure}[h]
\begin{tikzpicture}[scale=1.1]
\node at(4,2.2){\includegraphics[width=0.66\textwidth]{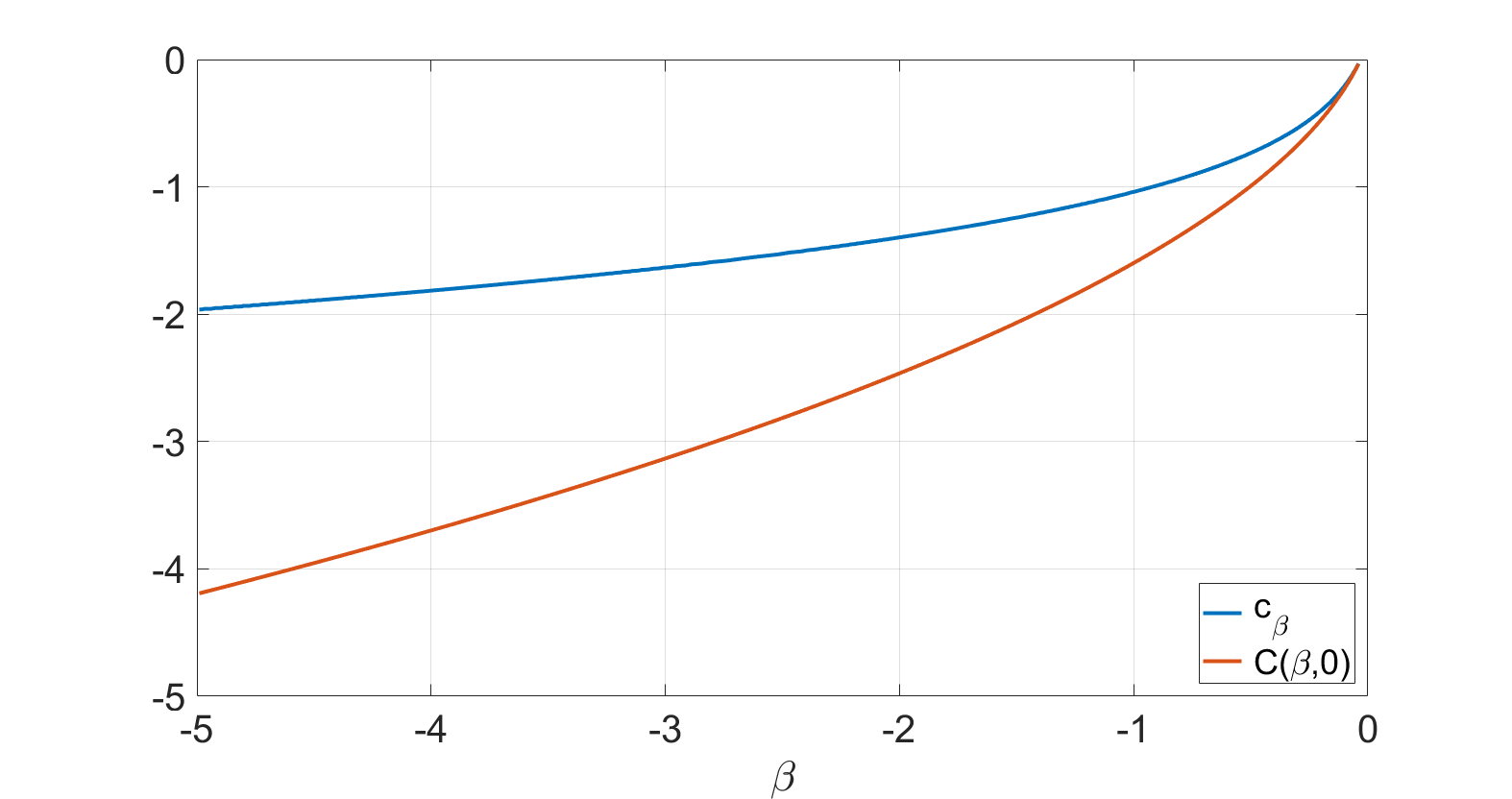}};
\draw[fill=white,white] (3.8,-0.4) rectangle (4.4,0.1);
\node at (4.2,-0.3) {$\beta$};
\draw[fill=white, white] (6.3,0.38) rectangle (7.8,1.1);
\node[blue!80] at (3,3.3) {$\ttc_{-\beta}$};
\node[red!80!blue] at (6,1.9) {$\ttC^\FHN(-\beta,0)$};
\end{tikzpicture}
\begin{minipage}[b]{0.3\textwidth}
\caption{The region $\beta\in [-5,-0.03]$ was sampled with $\beta$-steps of size
  $0.01$ to find $\ttc_{-\beta} \in [\ttC^\FHN(-\beta,0),0]$.\newline
\mbox{}}
\label{fig:speed}
\end{minipage}
\end{figure}
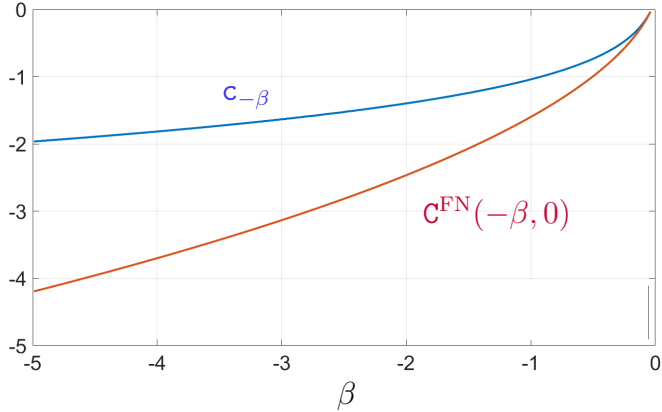

For $\beta \in [\beta_0(0.6),0]$ Figure \ref{fig:small.beta} shows
$0 \leq \ttc_{-\beta} \leq \ttC(-\beta,0)$ as predicted by the theory.  For
positive $\beta \in {]0,0.04[}$ Theorem \ref{th:MainResult} does not apply,
however, the comparison argument in \eqref{eq:relation.beta} predicts the
reverse relation $\ttc_{-\beta} \geq \ttC(-\beta,0) \geq 0$, which is confirmed
in the simulations, which also show that $\ttc_{-\beta}$ and $\ttC(-\beta,0)$
only differ less than a percent in relative size for all $\beta\in
[-0.06,0.03]$. Moreover, the identities $\ttc_{-\beta}=\ttC(-\beta,0)$ are
confirmed for $\beta=\beta_0(0.6)$ and $\beta=0$.
\begin{figure}[h]
{\centering 
\begin{tikzpicture}
\node at(4.8,1.433){
\includegraphics[width=0.79\textwidth,trim=18 50 20 30,clip=true]
                {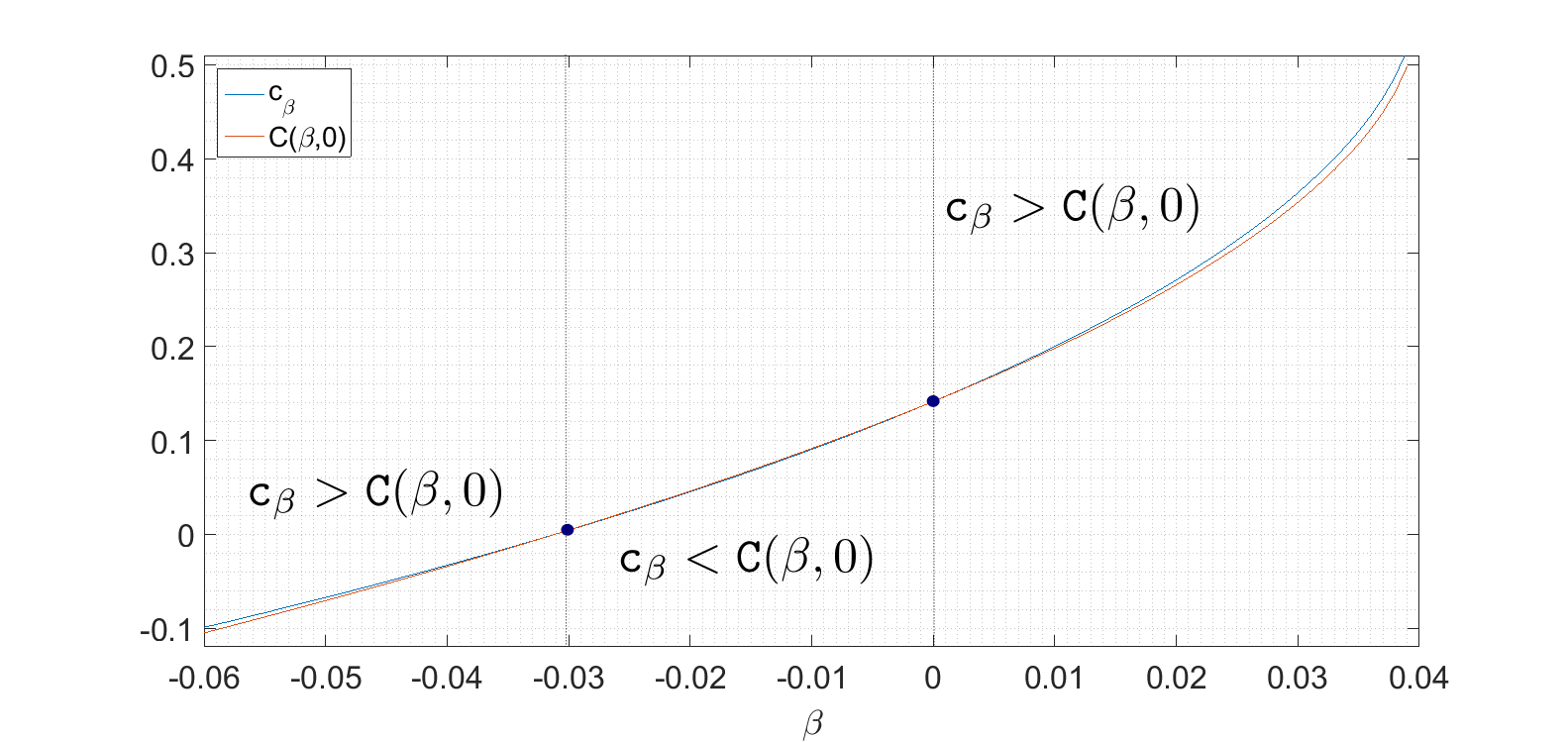}};
\draw[fill=green!10,opacity=0.0] (0,0) rectangle (10,3);
\draw[fill , color=white,opacity=0.99] (0.1,3.06) rectangle (1.3,3.9);
\draw[fill , color=white,opacity=0.99] (6.15,2.4) rectangle (8.3,2.9);
\node at (7.3,2.6) {{\scriptsize$\ttc_{-\beta} >\ttC^\FHN(-\beta,0)$}};
\draw[fill , color=white!100!blue,opacity=0.99] (3.45,-0.45) rectangle (5.55,-0.02);
\node at (4.5,-0.4) {{\scriptsize$\ttc_{-\beta} < \ttC^\FHN(-\beta,0)$}};
\draw[fill , color=white!100!blue,opacity=0.99] (0.4,0.1) rectangle (2.5,0.52);
\node at (1.45,0.3) {{\scriptsize$\ttc_{-\beta} >\ttC^\FHN(-\beta,0)$}};
\node[blue!80] at (9,3.3) {$\ttc_{-\beta}$};
\node[red!80!blue] at (9,1.8) {$\ttC^\FHN(-\beta,0)$};
\draw[very thin,->] (0,0) -- (11,0) node[above] {$\beta$};

\end{tikzpicture}\medskip\par
}

\begin{minipage}{0.52\textwidth}
\includegraphics[width=1.0\textwidth]{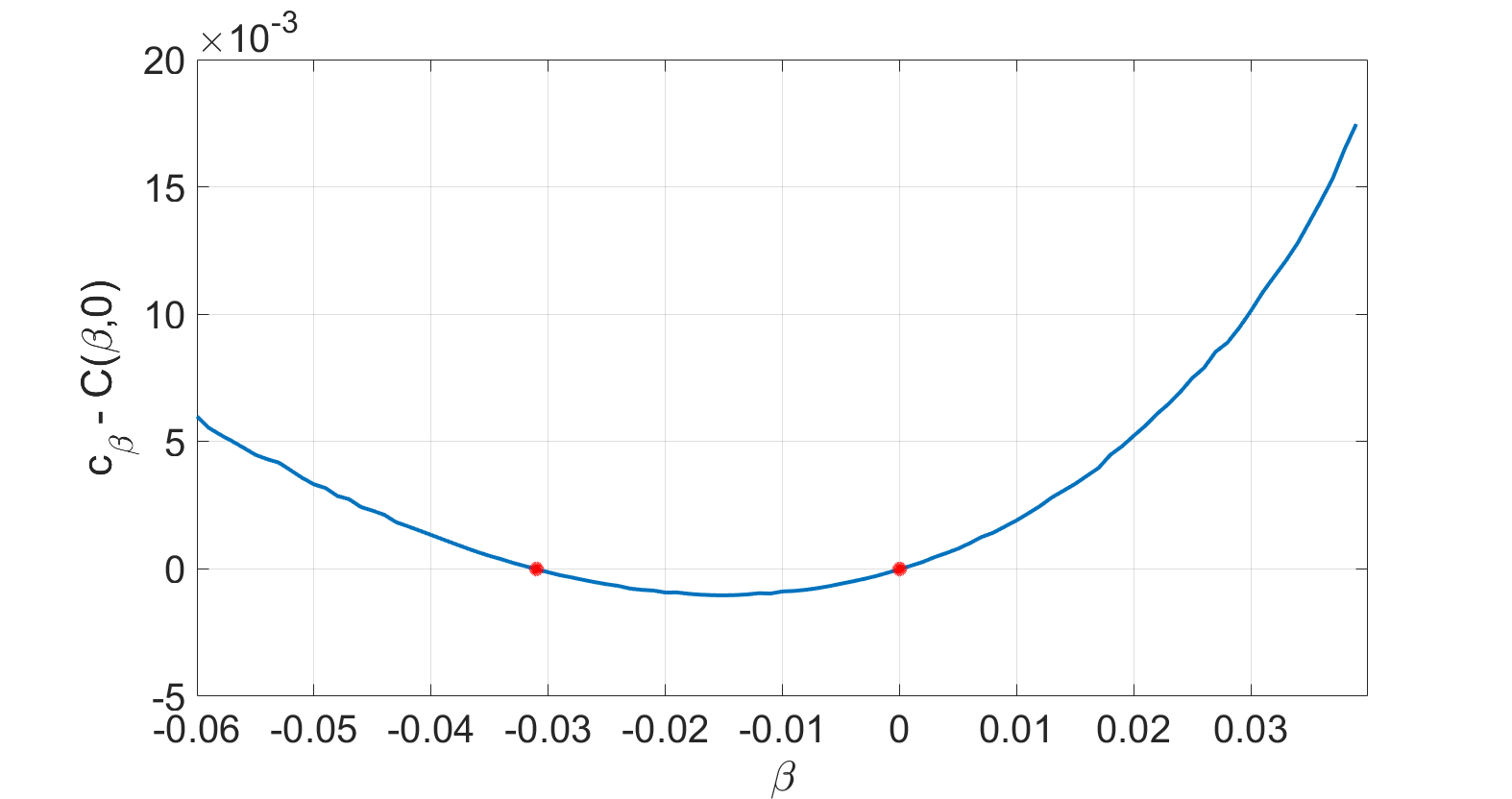}
\end{minipage}\hfill
\begin{minipage}{0.44\textwidth}
\caption{The region $\beta\in [-0.06,0.04]$ was sampled with $\beta$-steps of size
  $0.001$ to compare $\ttc_{-\beta}$ and $\ttC(-\beta,0)$ in the upper
  figure. As the two graphs differ only little, the lower figure displays
  $c_{-\beta} - \ttC(-\beta,0)$, where the sign changes at
  $\beta=\beta_0(0.6)\approx -0.0311$ and $\beta=0$ are clearly visible.\newline
\mbox{}}
\label{fig:small.beta}
\end{minipage}
\end{figure}

\appendix

\section{Two-scale convergence in weighted Sobolev spaces}
\label{sec:appendix-TS}

The aim of this appendix is to explain how the result in Theorem
\ref{thm:ts-conv-2} for $\Omega =\R^d$ can be derived by generalizing the
theory for bounded Lipschitz domains $\Omega \subset \R^d$ as developed in
\cite{MiReTh14TSHN, Reic15TSHS}. To this end we introduce a few notations from
two-scale convergence and two-scale homogenization as developed in
\cite{Ngue89GCRF, Alla92HTSC, MieTim07TSHE, CiDaGr08PUMH}.  For arbitrary
(bounded or unbounded) Lipschitz domains $\Omega \subseteq \bbR^d$, \EEE we set
\begin{align*}
	\calT_\eps : \rmL^2(\Omega) \to \rmL^2(\bbR^d\ti\bbT^d); \quad
	(\calT_\eps u)(x,y) := u_\mathrm{ex} \left( \eps \left[ \tfrac{x}{\eps}
          \right] + \eps y \right) , 
\end{align*}
where $u_\mathrm{ex} \in\rmL^2(\bbR^d)$ is obtained by extension with $0$ on
$\bbR^d \backslash \Omega$ and $[z] \in \bbZ^d$ denotes the integer part for
any $z \in \bbR^d$. We define \emph{weak and strong two-scale convergence} via
\begin{align*}
  u_\eps \wtc U \text{ in } \rmL^2(\Omega\ti\bbT^d) 
  & \quad : \overset{\text{Def.}}{\Longleftrightarrow} \quad \calT_\eps u_\eps
  \rightharpoonup U_\mathrm{ex} \text{ in } \rmL^2(\bbR^d\ti \bbT^d) , \\ 
  u_\eps \stc U \text{ in } \rmL^2(\Omega\ti\bbT^d) 
  & \quad : \overset{\text{Def.}}{\Longleftrightarrow} \quad \calT_\eps u_\eps
  \to U_\mathrm{ex} \text{ in } \rmL^2(\bbR^d\ti \bbT^d) . 
\end{align*}

We briefly recall the mathematical setting \AAA for the $\eps$-problem
\eqref{eq:SystOrig} such that \EEE the two-scale system
\eqref{eq:Syst} is rigorously justified. The fourth order diffusion tensors
$\bbD_i \in \mathrm{Lin}(\bbR^{d \ti m_i}; \bbR^{d \ti m_i})$ are uniformly
elliptic and bounded, i.e.\ there exist constants $M, \mu > 0$ such that
\begin{align*}
  \bbD_\rmi(y) \xi : \xi \geq \mu |\xi|^2
  \quad\text{and}\quad
  |\bbD_\rmi(y) \xi| \leq M |\xi| \quad\text{for all } \xi \in 
  \bbR^{d\ti m_\rmi} \text{ and } \rmi \in  \{\rmv,\rmw\},  
\end{align*}
and they are not necessarily symmetric. For the reaction terms we assume boundedness
$f_\rmi(\cdot,v,w) \in \rmL^\infty(\bbT^d)$ and global Lipschitz continuity, i.e.\
there exists $L>0$ such that for $v_1,v_2 \in \bbR^{m_\rmv}$ and
$w_1,w_2 \in \bbR^{m_\rmw}$
\begin{align*}
	|f_\rmi(y,v_1,w_1) - f_\rmi(y,v_2,w_2)| \leq L(|v_1 {-} v_2| + |w_1 {-} w_2|)
	\quad\text{for }  \rmi \in  \{\rmv,\rmw\} .
\end{align*}

\AAA In addition to the weighted Lebesgue spaces $\rmL^2_\varrho(\bbR^d)$ and
$ \rmL^2_\varrho(\bbR^d\ti \bbT^d)$ from \eqref{eq:WeightedL2}, \EEE we also
define the weighted Sobolev spaces $\rmH^1_\varrho(\bbR^d)$ and
$\rmL^2_\varrho(\bbR^d; \rmH^1(\bbT^d))$ via
\begin{align*}
  \rmH^1_\varrho(\bbR^d) & := \Bigset{ u \in \rmH^1_\mathrm{loc}(\bbR^d) }
      { \Vert u \Vert_{\rmL^2_\varrho(\bbR^d)}  +  \Vert \nabla u
      \Vert_{\rmL^2_\varrho(\bbR^d)} < \infty  }, 
  \\
  \rmL^2_\varrho(\bbR^d; \rmH^1(\bbT^d)) & := \Bigset{ U \in
    \rmL^2_\mathrm{loc}(\bbR^d; \rmH^1(\bbT^d)) }
      { \Vert U \Vert_{\rmL^2_\varrho(\bbR^d\ti\bbT^d)}  +  \Vert
      \nabla_y U \Vert_{\rmL^2_\varrho(\bbR^d\ti\bbT^d)} < \infty  }.
\end{align*} 
\AAA Our choice $\varrho(x)=1/\cosh(|x|/R)$ gives \EEE
$\varrho, \varrho^{-1} \in \rmL^1_\mathrm{loc}(\bbR^d)$, which implies that the
weighted Lebesgue and Sobolev spaces are Banach spaces, see e.g.\ 
\cite[Thm.\,1]{GolUkh09WSSE}. Moreover, introducing the scalar product
$(u,v)_{\rmL^2_\varrho(\bbR^d)} := (\sqrt{\varrho} u , \sqrt{\varrho} v
)_{\rmL^2(\bbR^d)}$ yields that they are also separable Hilbert spaces.
The proof of Theorem \ref{thm:ts-conv-2} follows along the lines of
\cite[Thm.\,4.1]{MiReTh14TSHN} and \cite[Thm.\,2.1.1,\,Thm.\,2.2.1]{Reic15TSHS}
with the following modifications.
\begin{itemize}
\item The definition of the periodic unfolding operator 
  \begin{align*}
    \calT_\eps : \rmL^2_\varrho(\bbR^d) \to \rmL^2_\varrho(\bbR^d \ti \bbT^d); \quad
    (\calT_\eps u)(x,y) := u \left( \eps \left[ \tfrac{x}{\eps} \right] + \eps y \right) 
  \end{align*}
  is immediate and the notion of two-scale convergence follows
  analogously. Notice that the assumptions on the weight function imply the
  unfolding estimate
  \begin{align*}
    \Vert \calT_\eps \varrho - \varrho \Vert_{\rmL^\infty(\bbR^d\ti\bbT^d)} 
    \leq \eps \sqrt{d} \Vert \nabla \varrho \Vert_{\rmL^\infty(\bbR^d)} 
    \leq \eps \sqrt{d} /R.
  \end{align*}
  Hence, we obtain the upper bound for unfolded functions
  $u \in \rmL^2_\varrho(\bbR^d)$
  \begin{align*}
    \Vert \calT_\eps u \Vert_{\rmL^2_\varrho(\bbR^d \ti \bbT^d)} \leq U_\eps
    \Vert u \Vert_{\rmL^2_\varrho(\bbR^d)} 
    \quad\text{with}\quad 
    U_\eps = \big( 1 - \eps \sqrt{d}/R \big)^{-1/2} .
  \end{align*}
	
\item The standard compactness results for two-scale convergence
  \cite{Ngue89GCRF,Alla92HTSC,MieTim07TSHE} also hold for weighted spaces,
  since they are, in particular, separable Banach spaces and the
  Banach--Alaoglu Theorem is applicable.

\item The definition of the gradient folding operator
  $\mathcal{G}^\mathrm{slow}_\eps$ (cf.\ \cite[Def.\,3.5]{MiReTh14TSHN}) needs
  to be adapted as follows: In the case of slow diffusion of order $O(\eps^2)$,
  we define for $W \in \rmL^2_\varrho(\bbR^d; \rmH^1(\bbT^d))$ the one-scale
  function $\mathcal{G}^\mathrm{slow}_\eps W := \hat{w}_\eps \in
  \rmH^1_\varrho(\bbR^d)$, which is given by \AAA the Lax--Milgram lemma \EEE 
  as the unique solution of the elliptic problem
\begin{align*}
\int_{\bbR^d} \big[ \sqrt{\varrho} \hat{w}_\eps - \mathcal{F}_\eps \left(
  \sqrt{\varrho} W\right) \big] \cdot \sqrt{\varrho} \varphi  
+ \big[ \eps \sqrt{\varrho} \nabla \hat{w}_\eps - \mathcal{F}_\eps \left(
  \sqrt{\varrho} \nabla_y W\right) \big] : \eps \sqrt{\varrho} \nabla \varphi
\dd x = 0  
\end{align*}
for all $\varphi \in \rmH^1_\varrho(\bbR^d)$. 

\item For the slowly diffusing species $w_\eps$, we choose the test function
  $\varphi^\mathrm{slow}_\eps = \varrho \left( w_\eps {-}
    \mathcal{G}^\mathrm{slow}_\eps W \right)$ and continue as in
  \cite{MiReTh14TSHN,Reic15TSHS}.  For species $v_\eps$ undergoing diffusion of
  order $O(1)$, we proceed analogously (cf.\ \cite[Def.\,1.2.7]{Reic15TSHS}) and
  choose the test function
  $\varphi^\mathrm{stand}_\eps = \varrho \left( v_\eps -
    \mathcal{G}^\mathrm{stand}_\eps V \right)$.
\end{itemize}
\AAA With these modifications for $\Omega =\R^d$, the proofs in
\cite[Thm.\,4.1]{MiReTh14TSHN} and \cite[Thm.\,2.1.1,\,Thm.\,2.2.1]{Reic15TSHS}
can be easily adapted, and the proof of Theorem \ref{thm:ts-conv-2} is
complete. \EEE

\paragraph*{Acknowledgments.} The authors are grateful to Christian K\"uhn for
illuminating discussions and to Stefanie Schindler for many helpful
comments. This research was supported by Deutsche Forschungsgemeinschaft
through the Collaborative Research Center SFB 910 \emph{Control of
  self-organizing nonlinear systems} (Project no.\ 163436311) via the
subproject A5 ``Pattern formation in coupled parabolic systems''.


\footnotesize

\def\cprime{$'$}

\end{document}